\documentclass{agtart_a}
\pdfoutput=1

\usepackage[all]{xy}

\title{Completed representation ring spectra of nilpotent groups}
\author{Tyler Lawson}
\givenname{Tyler}
\surname{Lawson}
\address{Department of Mathematics\\
Massachusetts Institute of Technology\\\newline
Cambridge MA 02139\\USA}
\email{tlawson@math.mit.edu}



\received{11 April 2005}
\revised{31 October 2005}
\accepted{5 January 2006}
\proposed{}
\seconded{}
\publishedonline{26 February 2006}
\published{26 February 2006}

\volumenumber{6}
\issuenumber{}
\publicationyear{2006}
\papernumber{8}
\startpage{253}
\endpage{285}

\doi{}
\MR{}
\Zbl{}

\arxivreference{}
\arxivpassword{}

\let\xysavmatrix\xymatrix
\def\xymatrix{\disablesubscriptcorrection\xysavmatrix}
\AtBeginDocument{\let\bar\wbar\let\tilde\wtilde\let\hat\what}
\newcommand{\we}{\smash{\rlap{\kern 6pt 
\raise 4pt\hbox{\footnotesize $\sim$}}}\longrightarrow}

\makeatletter
\def\cnewtheorem#1[#2]#3{\newtheorem{#1}{#3}
\expandafter\let\csname c@#1\endcsname\c@thm}

\newtheorem{thm}{Theorem}
\cnewtheorem{cor}[thm]{Corollary}
\cnewtheorem{prop}[thm]{Proposition}
\cnewtheorem{lem}[thm]{Lemma}
\theoremstyle{definition}
\cnewtheorem{defn}[thm]{Definition}
\theoremstyle{remark}
\cnewtheorem{rmk}[thm]{Remark}
\cnewtheorem{exam}[thm]{Example}
\cnewtheorem{case}[thm]{Case}

\makeatother  

\newcommand{\mb}[1]{\mathbb{#1}}
\newcommand{\mf}[1]{\mathfrak{#1}}

\newcommand{\Ext}{\ensuremath{{\rm Ext}}}

\newcommand{\Hom}{\ensuremath{{\rm Hom}}}
\newcommand{\Tor}{\ensuremath{{\rm Tor}}}

\newcommand{\colim}{\ensuremath{\mathop{\rm colim}}}
\newcommand{\hocolim}{\ensuremath{\mathop{\rm hocolim}}}
\newcommand{\holim}{\ensuremath{\mathop{\rm holim}}}
\newcommand{\overto}{\mathop\rightarrow}
\newcommand{\longoverto}{\mathop{\longrightarrow}}

\newcommand{\Irr}{\ensuremath{\mathop{\rm Irr}}}
\newcommand{\Rep}{\ensuremath{\mathop{\rm Rep}}}
\newcommand{\Map}{\ensuremath{{\rm Map}}}

\newcommand{\Uni}{{\rm U}}
\newcommand{\Sp}{{\cal S}p}
\newcommand{\Sym}{{\rm Sym}}

\newcommand{\eilm}[1]{\ensuremath{{\mb H} #1}}
\newcommand{\smsh}[1]{\ensuremath{\mathop{\wedge}_{#1}}}

\newcommand{\mapset}[3]{\ensuremath{\left[#2,#3\right]_{#1}}}
\newcommand{\compl}[1]{\ensuremath{#1^\wedge}}
\newcommand{\loca}[3]{\ensuremath{L^{#1}_{#2}(#3)}}

\begin{document}

\begin{htmlabstract}
In this paper, we examine the "derived completion" of the
representation ring of a pro&ndash;p group G<sub>p</sub>^ with
respect to an augmentation ideal.  This completion is no longer a
ring: it is a spectrum with the structure of a module spectrum over
the Eilenberg&ndash;MacLane spectrum <b>HZ</b>, and can have higher
homotopy information.  In order to explain the origin of some of these
higher homotopy classes, we define a deformation representation ring
functor R[&ndash;] from groups to ring spectra, and show that the map
R[G<sub>p</sub>^]&rarr;R[G] becomes an equivalence after
completion when G is finitely generated nilpotent.  As an
application, we compute the derived completion of the representation
ring of the simplest nontrivial case, the p&ndash;adic Heisenberg group.
\end{htmlabstract}

\begin{asciiabstract}
In this paper, we examine the `derived completion' of the
representation ring of a pro-p group G_p^ with respect to an
augmentation ideal.  This completion is no longer a ring: it is a
spectrum with the structure of a module spectrum over the
Eilenberg-MacLane spectrum HZ, and can have higher homotopy
information.  In order to explain the origin of some of these higher
homotopy classes, we define a deformation representation ring functor
R[-] from groups to ring spectra, and show that the map R[G_p^] -->
R[G] becomes an equivalence after completion when G is finitely
generated nilpotent.  As an application, we compute the derived
completion of the representation ring of the simplest nontrivial case,
the p-adic Heisenberg group.
\end{asciiabstract}

\begin{abstract}
In this paper, we examine the ``derived completion'' of the
representation ring of a pro-$p$ group ${\cal G}_p^\wedge$ with
respect to an augmentation ideal.  This completion is no longer a
ring: it is a spectrum with the structure of a module spectrum over
the Eilenberg--MacLane spectrum $\mathbb{HZ}$, and can have higher
homotopy information.  In order to explain the origin of some of these
higher homotopy classes, we define a deformation representation ring
functor $R[-]$ from groups to ring spectra, and show that the map
$R[{\cal G}_p^\wedge] \to R[{\cal G}]$ becomes an equivalence after
completion when ${\cal G}$ is finitely generated nilpotent.  As an
application, we compute the derived completion of the representation
ring of the simplest nontrivial case, the $p$--adic Heisenberg group.
\end{abstract}

\maketitle

\section{Introduction}
The original Adams spectral sequence provided a streamlined method for
computing the $p$--primary part of the stable homotopy groups of
spheres, or more generally for attempting to compute the group of maps
between two objects in the stable homotopy category.  This process has
seen highly useful generalizations, such as the Adams--Novikov spectral
sequence exposited in Adams \cite{adams74}, and the $E$--nilpotent
completion spectral sequence of Bousfield \cite{bousfield79}.  By
using the same method of proof, Carlsson gave a construction for
modules over a ring $R$ in \cite{carlsson78}.  The underlying method
is that given an object $X$ and a generalized homology theory $E$, one
constructs a ``derived completion'', or $E$--nilpotent completion,
which is a pro-object $\compl{X}_E$, and there is a spectral sequence
for computing the space of maps from an object $Y$ into $\compl{X}_E$.
In the case where $X$ is a finitely generated module over a Noetherian
ring and $E$ is the homology theory $\Tor^R_*(R/I,-)$, the completion
$\compl{X}_E$ is equivalent to the $I$--adic completion of $X$.

With the advent of highly structured categories of modules over
algebra spectra, such as those constructed by
Elmendorf--Kriz--Mandell--May \cite{ekmm97}, it has become possible to
construct an Adams spectral sequence generalizing all of the above
constructions, as in Baker--Lazarev \cite{bl01}.  For example, by
\cite[Theorem~IV.2.4]{ekmm97}, the derived category of modules over a
ring $R$ is equivalent to the homotopy category of module spectra over
the Eilenberg--MacLane spectrum $\eilm{R}$.  (Throughout this paper, we
will use the notation $\eilm{X}$ for the Eilenberg--MacLane spectrum
associated to $X$.)

In this paper, we will be mainly focused on the case where $M = R$, as
an analogue of completing a ring with respect to an ideal with
quotient ring $E$.  When $E$ is an $R$--algebra, the
$E$--nilpotent completion of $R$ (as an $R$--module) can be canonically
expressed as the totalization of the cobar complex
\[
\{ E \smsh{R} \cdots \smsh{R} E \},
\]
with coface and codegeneracy maps given by the unit map $R \to E$ and
multiplication $E \smsh{R} E \to E$ respectively.  (A different
characterization, and another canonical construction, will be given in
\fullref{sec:completions}.)  We write $\compl{R}_E$ for this ``derived
completion'' of $R$ as an $R$--module spectrum.

Our interest is in examining this completion functor in some specific
cases of modules over a non-Noetherian ring.  Let $R$ be a ring, and
$E$ and $M$ chain complexes of left $R$--modules.  Then the completion
$\compl{M}_E$ is a pro-object in the derived category of $R$--modules,
or equivalently a pro-weak equivalence class of chain complexes.
Under certain algebraic conditions, including $\Tor^R_*(E,E)$ being
flat over $H_* E$, there is a spectral sequence
\begin{equation}
  \label{eq:adamsss}
\Ext^{**}_{\Ext_R^*(E,E)} (\Ext_R^*(M,E), H^*(E))
\end{equation}
that abuts to the homology groups of $\compl{M}_E$.  Convergence is a
more delicate question; for example, see Baker--Lazarev \cite{bl01}.

One of the first interesting examples has to do with the relationship
between circle actions and actions of finite cyclic groups.
Specifically, let $\mb Z_p$ be the group of $p$--adic integers under
addition, and let $R[\mb Z_p]$ be its representation ring.  This is
the free abelian group on irreducible continuous complex
representations of $\mb Z_p$, and as such is isomorphic to the group
ring $\mb Z[\mu_{p^\infty}]$, the group ring on the group
$\mu_{p^\infty}$ of $p$--power roots of unity in $\mb C^\times$.  If
$E$ is the $R[\mb Z_p]$--module $\mb Z/p$ with trivial action, one can
compute with the above spectral sequence and find that
\begin{equation}
\label{eq:circlecase}
H_*(\compl{R[\mb Z_p]}_{\mb Z/p}) \cong 
\begin{cases}
\mb Z_p & \hbox{ if } * = 0,1\\
0 & \hbox{ otherwise.}
\end{cases}
\end{equation}
The computation of the completion of the representation ring of the
$p$--adic Heisenberg group $\compl{H}_p$ with respect to the
$R[\compl{H}_p]$--module $\mb Z/p$ was initiated by C Bray in
\cite{bray99}, and completed here.  The result takes the following
form:
\begin{equation}
\label{eq:heiscase}
H_*\left(\compl{R[\compl{H}_p]}_{\mb Z/p}\right) \cong 
\begin{cases}
\widehat \oplus_{\zeta \in \mu_{p^\infty}} \mb Z_p & \hbox{ if } * = 0 \\
\widehat \oplus_{\zeta \in \mu_{p^\infty}} \mb Z_p \oplus \mb Z_p &
\hbox{ if } * = 1 \\
\widehat \oplus_{\zeta \in \mu_{p^\infty}} \mb Z_p & \hbox{ if } * = 2 \\
0 & \hbox{ otherwise.}
\end{cases}
\end{equation}
(Here $\hat \oplus$ denotes the completed direct sum of $p$--complete
modules.)  This computation can be carried out using the spectral
sequence of equation~\ref{eq:adamsss}.  As such, it involves a
significant number of extension problems, especially with respect to
determining an algebra structure possessed by
$\compl{R[\compl{H}_p]}_{\mb Z/p}$.

The purpose of this paper is to explore the origin of higher homology
classes in the case of $R[\compl{{\cal G}}_p]$, where ${\cal G}$ is a
finitely generated nilpotent group and $\compl{{\cal G}}_p$ is its
pro-$p$ completion.  Our contention is that some of this higher
homotopy arises from {\em geometric\/} information about the
representation theory of ${\cal G}$.

Specifically, let $\Irr({\cal G})$ be the set of isomorphism classes
of irreducible representations of ${\cal G}$.  $\Irr({\cal G})$ comes
naturally equipped with the structure of a topological space, and
there is a natural map $\Irr(\compl{{\cal G}}_p) \to \Irr({\cal G})$,
where $\Irr(\compl{{\cal G}}_p)$ has the discrete topology.  (Both of
these topologies arise naturally from the starting point of the
compact--open topology on mapping spaces.)  Applying $\eilm{\mb Z}
\smsh{} \Sigma^\infty_+(-)$, where $\Sigma^\infty_+$ denotes the
unbased suspension spectrum functor, gives a map
\begin{equation}
\label{eq:mainmap}
\eilm{R[\compl{{{\cal G}}}_p]} \to 
\eilm{\mb Z} \smsh{} \Sigma^\infty_+ \Irr({\cal G})
\end{equation}
of $\eilm{\mb Z}$--modules.

\begin{prop}
\label{prop:mainmap}
If ${\cal G}$ is a finitely generated, nilpotent, discrete group,
there exists a natural map of $\mb S$--algebras realizing the map of
equation~\ref{eq:mainmap}.
\end{prop}

\begin{proof}
In \fullref{sec:repringspectra}, we construct a map of commutative
monoids in symmetric spectra $R[\compl{{\cal G}}_p] \to R[{\cal G}]$,
which can be realized as a map of $\mb S$--algebras.
\fullref{sec:compiso}, \fullref{prop:repdga} and
\fullref{cor:finiterepring} show this is equivalent to the map
of equation~\ref{eq:mainmap} in the case where ${\cal G}$ is finitely
generated, nilpotent, and discrete.
\end{proof}

Our main result is the following.
\begin{thm}
\label{thm:main}
The map 
$\eilm{R[\compl{{\cal G}}_p]} \to \eilm{\mb Z} \smsh{}
\Sigma^\infty_+ (\Irr({\cal G}))$ 
induces an equivalence of derived completions
\[
\compl{\eilm{R[\compl{{\cal G}}_p]}}_{\eilm{\mb F_p}} \to
\compl{\eilm{\mb Z} \smsh{} \Sigma^\infty_+ (\Irr({\cal
G}))}_{\eilm{\mb F_p}},
\]
where the derived completion on the left is taken as modules over
$\eilm{R[\compl{{\cal G}}_p]}$ and on the right as modules over
$\eilm{\mb Z} \smsh{} \Sigma^\infty_+(\Irr({\cal G}))$.
\end{thm}

The proof requires some intermediate results.  We first make the
following definition.

\begin{defn}
\label{def:multequiv}
Suppose $E$ is a fixed commutative connective $\mb S$--algebra.  A map
$A \to A'$ of commutative $\mb S$--algebras over $E$ is a multiplicative
$E$--equivalence if the map $E \to A' \smsh{A} E$ is a equivalence.
\end{defn}

In \fullref{sec:completions} we will recall some definitions and
properties of Bousfield localization.  A multiplicative
$E$--equivalence $A \to A'$ is also an ordinary $E$--equivalence of
$A$--modules, and hence $A$ and $A'$ have the same $E$--localizations as
$A$--modules.  More is true, however; by
\fullref{prop:compiso}, the $E$--localization of $A$ as an
$A$--module is homotopy equivalent to the $E$--localization of $A'$ as an
$A'$--module.

The $E$--localization $\loca{A}{E}{A}$ of $A$ and the derived
completion $\compl{A}_E$ are related by a map $\loca{A}{E}{A} \to
\compl{A}_E$, where the localization is viewed as a constant
pro-object.  The $E$--localization of $A$ is equivalent to the
homotopy limit $\holim(\compl{A}_E)$ under good circumstances, such as
when smashing with $E$ can be passed under the limit.

\begin{thm}
\label{thm:main2}
Suppose ${\cal G}$ is a finitely generated discrete group.
The map $\eilm{R[\compl{{\cal G}}_p]} \to \eilm{\mb Z} \smsh{}
\Sigma^\infty_+ (\Irr({\cal G}))$ is a multiplicative $\eilm{\mb
F_p}$--equivalence.
\end{thm}

\begin{proof}
By \fullref{prop:repdga} and
\fullref{cor:finiterepring}, this statement is equivalent to
\fullref{prop:main2}.
\end{proof}

The results leading to \fullref{thm:main2} are proved by making use
of the structure of these $\mb S$--algebras as algebra spectra over
$R[\compl{{\cal G}}_{\rm ab}]$, the representation ring of the
abelianization of the profinite completion of ${\cal G}$.

\fullref{thm:main} then follows from \fullref{thm:main2} and
\fullref{lem:boustocomp}, which proves that multiplicative
$E$--equivalences induce equivalences of derived completions.  In fact,
the groups listed in equation~\ref{eq:heiscase} will be proven in
\fullref{sec:heiscomp} to be the homotopy groups of the $\eilm{\mb
F_p}$--localization of $R[\compl{H}_p]$.

We now briefly outline the stages of the argument.  In
\fullref{sec:completions} we will show that multiplicative
$E$--equivalences are closed under composition in the category of $\mb
S$--algebras over $E$. Additionally, we show in
\fullref{prop:compiso} that if $R$ is a commutative $\mb
S$--algebra, and $A \to A'$ is a map of commutative $R$--algebras over
$E$ which is an $E$--equivalence as $R$--modules, then this map is also
a multiplicative $E$--equivalence.

In \fullref{sec:nilrep} we recall and apply results of Lubotzky--Magid
\cite{lubotzkymagid85} about the spaces of representations associated
to nilpotent groups.  \fullref{sec:repringspectra} models the map of
equation~\ref{eq:mainmap} by a naturally defined map of ring objects
in $\Gamma$--spaces, obtained by considering spaces of representations.
The proof of \fullref{thm:main} is the main content of
\fullref{sec:compiso}.

The $\mb S$--algebra $\eilm{\mb Z} \smsh{} \Sigma^\infty_+ (\Irr({\cal
G}))$ is more tractable.  In \fullref{sec:heiscomp}, this spectrum
is used to compute the homotopy groups of the Bousfield localization
of the representation ring of $\compl{H}_p$, the $p$--adic Heisenberg
group.  The results agree with the computation referenced in
equation~\ref{eq:heiscase}, with the additional benefit that the
algebra structure becomes apparent.  The completion of this
computation appears in \fullref{sec:heisdone}: the homology of the
Bousfield localization of $R[\compl{H}_p]$ is a retract of the
$p$--completion of $H_*(\Irr(H))$, the homology of the space of
irreducible representations of the integral Heisenberg group.  The
retract contains precisely those summands of the homology
corresponding to components that contain pro-$p$ representations.
Finally, in \fullref{sec:heisnilp}, these groups are shown to
coincide with the homology groups of the derived completion
$\compl{R[\compl{H}_p]}_{\eilm{\mb F_p}}$.

\medskip
{\bf Acknowledgements}\qua  The author would like to thank Daniel Bump,
Veronique Godin, Robert Guralnick, Haynes Miller, and Daniel Ramras
for helpful discussions related to this paper.  The majority of this
research was carried out while the author was a student of Gunnar
Carlsson, and the present work owes much to him.  The author was
partially supported by NSF award 0402950 and the ARCS foundation.

\section{Localizations of module spectra}
\label{sec:completions}

Throughout this section, we will be working in the category ${\cal
  M}_R$ of module spectra over a fixed commutative $\mb S$--algebra
  $R$, and $E$ will denote an element of ${\cal M}_R$.  (In the
  applications, $E$ will be a ``quotient'' $R$--algebra with unit $\eta
  \co R \to E$.)  See Elmendorf--Kriz--Mandell--May \cite{ekmm97} for
  foundational material on these categories of spectra.  Our real
  interest lies in the homotopy category ${\mathcal D}_R$, which is
  closed symmetric monoidal under $\smsh{R}$ and $F_R(-,-)$, and so we
  will usually replace $R$--modules by cell $R$--modules without
  comment.  The book Hovey--Palmieri--Strickland \cite{hps97} serves as
  a good introduction to the abstract theory of these homotopy
  categories.

We will write $M_*$, $M_*(N)$, and $\mapset{}{M}{N}^*$ for the groups
$\pi_*(M)$, $\pi_*(M \smsh{R} N)$, and $\pi_{-*}(F_R(N,M))$
respectively, where $F_R(N,M)$ is the $R$--module function spectrum.
(The underlying $\mb S$--algebra $R$ will be specified if it is
ambiguous.)

In the setting of ring spectra, one appropriate replacement for the
notion of ``completion at an ideal'' is Bousfield localization.  The
$E$--localization of an $R$--module $M$ is obtained from $M$ by throwing
away all information that is not determined by the homology theory
$E_*(-)$.

We begin by recalling some definitions.

\begin{defn}
An $R$--module $X$ is $E$--acyclic if $E \smsh{R} X \simeq *$.
\end{defn}
\begin{defn}
A map $f\co M \to N$ of $R$--modules is an $E$--equivalence if the cofiber
is $E$--acyclic.
\end{defn}

\begin{defn}
An object $M$ is $E$--local if $\mapset{R}{X}{M}^* = 0$ whenever $X$ is
$E$--acyclic.  ($E$--acyclic objects are closed under suspension, so
this is equivalent to requiring that $\mapset{R}{X}{M} = 0$ for all
$E$--acyclic $X$.)
\end{defn}

\begin{rmk}
The definition immediately implies that $E$--local objects form a
thick subcategory of ${\mathcal D}_R$ closed under homotopy limits.
If $E$ is an $R$--algebra, then the adjunction $\mapset{E}{E \smsh{R}
X}{M} \cong \mapset{R}{X}{M}$ implies that all $E$--module spectra
$M$ are $E$--local.
\end{rmk}

\begin{defn}
An $E$--localization of $M$ is an $E$--local object $\loca{}{E}{M}$
equipped with an $E$--equivalence $M \to \loca{}{E}{M}$.  If the
underlying $\mb S$--algebra $R$ is ambiguous, we write $\loca{R}{E}{M}$.
\end{defn}

The definition implies that any map from $M$ to a local object $N$
factors uniquely (up to homotopy) through $\loca{}{E}{M}$, and hence
any such localization is unique up to equivalence.  Therefore, an
$E$--equivalence $M \to M'$ gives rise to an equivalence
$\loca{}{E}{M} \to \loca{}{E}{M'}$.

For details about the existence of localizations, consult
\cite[Chapter~VIII]{ekmm97}.  A model structure exists on the category
of $R$--modules that has $E$--equivalences as weak equivalences, and the
fibrant objects in this model structure become $E$--local in the
homotopy category of $R$--modules.  For our purposes, however, only
existence is important.

We will now list some consequences about maps of $\mb S$--algebras and
their behavior under localization.

\begin{lem}
\label{lem:extcomp}
Suppose $R \to R'$ is a map of commutative $\mb S$--algebras, and $E$ is an
$R$--module.  Define $F = E \smsh{R} R'$.  Then for any $R'$--module
$M$, the map $M \to \loca{R'}{F}{M}$ is an $E$--localization of $M$ as
an $R$--module.
\end{lem}

\begin{proof}
This is \cite[Proposition~VIII.1.8]{ekmm97}.
\end{proof}

\begin{lem}
\label{lem:extacyclic}
Suppose $R \to R'$ is a map of commutative $\mb S$--algebras and
$E'$ is an $R'$--module.  If $M$ is an $R'$--module that is $E'$--acyclic
as an $R$--module, then $M$ is $E'$--acyclic as an $R'$--module.
\end{lem}

\begin{proof}
Without loss of generality, assume $R$ is a $q$--cofibrant commutative
$\mb S$--algebra, $R'$ is a commutative $q$--cofibrant $R$--algebra, and
$E'$ is a cell $R'$--module.  By \cite[Proposition~IX.2.3]{ekmm97}, $E'
\smsh{R'} M$ is equivalent to $B^R(E',R',M)$, the geometric
realization of the simplicial spectrum
\[
B^R_p(E',R',M) = E' \smsh{R} (R')^{\smsh{R} (p)} \smsh{R} M.
\]
The symmetric monoidal properties of the smash product in ${\cal M}_R$
imply that\break $B^R_p(E',R',M) \simeq *$ because $E' \smsh{R} M \simeq *$.
Therefore, $E' \smsh{R'} M \simeq *$.
\end{proof}

\begin{prop}
\label{prop:compiso}
Suppose there is a diagram $R \to A \to A'$ of maps of commutative
$\mb S$--algebras, and $E$ is an $A'$--module.  There is a natural (up
to homotopy) map $\loca{A}{E}{A} \to \loca{A'}{E}{A'}$.  If $A \to A'$
is an $E$--equivalence as $R$--modules, then this map is an equivalence.
\end{prop}

\begin{proof}
The object $\loca{A'}{E}{A'}$ is $E$--local as an $A$--module; if $X$ is
an $E$--acylic $A$--module, $A' \smsh{A} X$ is an $E$--acyclic
$A'$--module, and hence
\[
\mapset{A}{X}{\loca{A'}{E}{A'}} \cong \mapset{A'}{A' \smsh{A} X}
{\loca{A'}{E}{A'}} = 0.
\]
Therefore, the composite map $A \to A' \to \loca{A'}{E}{A'}$ factors
through $\loca{A}{E}{A}$.

Suppose the map $A \to A'$ is an $E$--equivalence of $R$--modules, or
equivalently the cofiber is $E$--acyclic as an $R$--module.  By
\fullref{lem:extacyclic}, the cofiber is $E$--acyclic as an
$A$--module, so without loss of generality we can assume $R = A$.

Define $F = E \smsh{A} A'$.  Because the map $A \to A'$ is an
$E$--equivalence, the maps
\[
E \smsh{A} A \to E \smsh{A} A' \to E
\]
are all equivalences, where the right-hand map is the
multiplication map for $A'$--modules.  This is an $A'$--module
map, so $F \simeq E$ as an $A'$--module.  Thus
\[
\loca{A}{E}{A} \simeq \loca{A}{E}{A'} \simeq \loca{A'}{F}{A'} \simeq
\loca{A'}{E}{A'},
\]
by \fullref{lem:extcomp}.
\end{proof}

\begin{lem}
\label{lem:locawequiv}
The collection of commutative $\mb S$--algebras over $E$, with
multiplicative $E$--equivalences as maps, forms a category.
\end{lem}

\begin{proof}
Suppose $A'' \to A \to A'$ are both multiplicative $E$-equivalences.
Consider the following diagram of maps.
$$\xymatrix{
E \ar[r] \ar[dr] &
(E \smsh{A''} A) \smsh{A} A' \ar[r] \ar[d] &
E \smsh{A''} A' \\
& E \smsh{A} A'
}$$
The lower maps are equivalences by assumption, while the rightmost map
is an isomorphism by associativity of the smash product.
\end{proof}

Let $R$ be a commutative $\mb S$--algebra and $E$ an $R$--algebra.  In
\cite{bousfield79}, the $E$--nilpotent completion (or derived
completion) of an ordinary spectrum was defined; we will briefly
recall this definition in the context of modules over $R$.

\begin{defn}
The family of $E$--nilpotent objects is the thick subcategory of the
category of $R$--modules generated by $E$--modules, i.e. the smallest
full subcategory containing the $E$--modules closed under retracts
and exact triangles.
\end{defn}

\begin{defn}
An $E$--nilpotent completion of an $R$--module $M$ is an inverse system 
$\{W_s\}_{s \geq 0}$ of $E$--nilpotent objects under $M$ such that
the map
\[
\lim \mapset{R}{W_s}{X} \to \mapset{R}{M}{X}
\]
is an isomorphism for any $E$--nilpotent object $X$.  (The
$E$--nilpotent completion of $M$ is defined up to pro-equivalence.)
\end{defn}

\begin{rmk}
Note that in order to check that an object is an $E$--nilpotent
completion of $M$, it suffices to check the isomorphism 
$\lim \mapset{R}{W_s}{X} \to \mapset{R}{M}{X}$ when $X$ is an
$E$--module, as both sides preserve retracts and exact triangles.
\end{rmk}

\begin{rmk}
$E$--nilpotent completions exist; they can be explicitly constructed
by the following standard procedure.  Let $I$ be the fiber of the
map $R \to E$, and $I^{(s)}$ its $s$--fold smash power over $R$.  Let
$R/I^{(s)}$ be the cofiber of the map $I^{(s)} \to R^{(s)} \cong R$.
One finds that there are cofiber
sequences
\[
E \smsh{R} I^{(s)} \smsh{R} M \to R/I^{(s+1)} \smsh{R} M \to
R/I^{(s)} \smsh{R} M
\]
for all $s \geq 0$.
By induction, we find that the $R/I^{(s)} \smsh{R} M$ are
$E$--nilpotent.  In addition, if $X$ is an $E$--module, the adjunction
\[
\mapset{R}{I^{(s)} \smsh{R} M}{X} \cong
\mapset{E}{E \smsh{R} I^{(s)} \smsh{R} M}{X}
\]
shows that any map in $\mapset{R}{R/I^{(s)}\smsh{R} M}{X}$ that goes to
zero in $\mapset{R}{R}{M}$ factors through $\mapset{R}{\Sigma E \smsh{R}
I^{(s)} \smsh{R} M}{X}$, which goes to zero in
$\mapset{R}{R/I^{(s+1)} \smsh{R} M}{X}$.  Additionally, the map
$\mapset{R}{R/I}{X} \to \mapset{R}{R}{X}$ is already surjective by the
same adjunction.
\end{rmk}

\begin{lem}
\label{lem:boustocomp}
If $M \to N$ is an $E$--equivalence of $R$--modules, then the map of
derived completions $\compl{M}_E \simeq \compl{N}_E$ is an equivalence
of pro-objects.
\end{lem}

\begin{proof}
Let $\{W_s\}$ be an $E$--nilpotent completion of $N$.
If $X$ is an $E$--module, then $X$ is $E$--local, so the composite map
\[
\lim \mapset{R}{W_s}{X} \to \mapset{R}{N}{X} \to \mapset{R}{M}{X}
\]
is an isomorphism.
\end{proof}

\section{Representations of nilpotent groups}
\label{sec:nilrep}
Let ${\cal G}$ be a topological group, and let $\Hom({\cal
G},\Uni(n))$ be the space of (continuous) group homomorphisms from
${\cal G}$ to $\Uni(n)$, considered as a subspace of the mapping space
$\Map({\cal G},\Uni(n))$.  The latter space is given the compact--open
topology.

\begin{exam}[Finitely generated discrete groups]
Suppose ${\cal G}$ is a finitely generated discrete group, with presentation
\[
{\cal G} = \left\langle g_1, g_2, \ldots \mid  r_1(g_i) = r_2(g_i) = \ldots = 1
\right\rangle.
\]
Then $\Hom({\cal G},\Uni(n))$ is the set of real points of an
algebraic variety, as follows.  Writing a general complex matrix in
the form $A = (a_{jk} + b_{jk}i)_{jk}$, we can expand out the formula
$A A^* = I$.  We find that the points of $\Uni(n)$ are the real points
of the algebraic variety
\[
\left\{(a_{jk},b_{jk}) \ \Big|\ \sum_\ell a_{j\ell}a_{k\ell} + b_{j\ell}b_{k\ell} =
\delta_{jk}, \sum_\ell b_{j\ell}a_{k\ell} - a_{j\ell}b_{k\ell} = 0 \right\}.
\]
The space $\Hom({\cal G}, \Uni(n))$ can be written as the set
\[
\left\{(A_1, A_2, \ldots) \in \prod \Uni(n) \Big| r_1(A_i)
  = r_2(A_i) = \ldots = I \right\},
\]
and the mapping topology coincides with the subspace topology from
$\prod \Uni(n)$.  (Pointwise convergence of homomorphisms is the same
as pointwise convergence on the generating set.)  The conditions
$r_j(A_i) = I$ introduce extra equations that the points must
satisfy, but these are still algebraic.  Therefore, the space
$\Hom({\cal G}, \Uni(n))$ is a closed subvariety of $\prod \Uni(n)$.
\end{exam}

\begin{exam}[Profinite groups]
\label{exam:profinite}
First, we will prove that there exists a sufficiently small open
neighborhood ${\cal U}$ of the identity in $\Uni(n)$ that contains no
nontrivial subgroups.  Let ${\cal V}$ be a ball centered at the
identity of the Lie algebra $\mf u(n)$ such that $\exp \co \mf u(n)
\to \Uni(n)$ is injective on $2{\cal V}$, and let ${\cal U} =
\exp({\cal V})$.  For any $I \neq g \in {\cal U}$, $g = \exp(A)$ for
some $A \in {\cal V}$.  $A \neq 0$, so there exists a unique
nonnegative integer $n$ such that $2^{n-1} A \in {\cal V}$ but
$2^{n}A \not \in {\cal V}$.  Then, because $2^n A \in 2{\cal V}$ where
the exponential map is injective, we find that $g^{2^n} = \exp(2^n A)$
is not in ${\cal U}$.

Given any normal open subgroup $N$ of ${\cal G}$, the set $D$
consisting of those homomorphisms mapping the compact set $N$ to the
open set $\mathcal{U}$ is precisely the set of homomorphisms factoring
through ${\cal G}/N$, because any such homomorphism must send $N$ to
the unique subgroup of ${\cal U}$.  This set $D$ is therefore both
open (in the compact--open topology) and closed (it is also the set of
maps sending $N$ to the identity).

Every continuous homomorphism from ${\cal G}$ to $\Uni(n)$ factors
through a finite quotient, as the image must be compact and because of
what was proved above.  As a result, $\Hom({\cal G},\Uni(n))$ is the
colimit of spaces $\Hom({\cal G}/N, \Uni(n))$ as $N$ varies over open
normal subgroups of ${\cal G}$.  Each of these spaces is an open and
closed subset of the colimit.  All of these are representation spaces
of finite groups, and as a result they are compact.
\end{exam}

\begin{rmk}
In both cases, $\Hom({\cal G},\Uni(n))$ is a Hausdorff space
and locally path connected.
\end{rmk}

The space $\Hom({\cal G},\Uni(n))$ has a continuous action of $\Uni(n)$ on it by
conjugation.  Define
\[
\Rep({\cal G}) = \coprod_{n \geq 0} \Hom({\cal G},\Uni(n)) / \Uni(n).
\]
Because $\Uni(n)$ is path connected and compact, in the discrete and
profinite cases $\Rep({\cal G})$ is a Hausdorff locally path
connected space.  However, for a general group ${\cal G}$ there can be
some point-set level pathology in this construction.

$\Rep({\cal G})$ parameterizes isomorphism classes of continuous
unitary representations of ${\cal G}$.  It has a continuous
augmentation map $\varepsilon\co \Rep({\cal G}) \to \mb N$ given by
sending a representation to its dimension.  Write ${\cal G}^*$ for the
preimage of $1$ under this map.

After making a choice of basis, there are continuous maps $\oplus\co
\Uni(n) \times \Uni(m) \to \Uni(n+m)$ and $\otimes\co \Uni(n) \times
\Uni(m) \to \Uni(nm)$ that induce continuous operations on $\Rep({\cal
G})$.  These operations on $\Rep({\cal G})$ are strictly
commutative, associative, and distributive, and the augmentation
$\varepsilon$ takes $\oplus$ to addition and $\otimes$ to
multiplication.

If $f \co {\cal G} \to {\cal G}'$ is a (continuous) group homomorphism,
there is a continuous restriction map $f^* \co \Rep({\cal G}') \to
\Rep({\cal G})$ preserving $\oplus$, $\otimes$, and $\varepsilon$.
$\Rep(-)$ is a contravariant functor to topological semirings, and
$(-)^*$ is a contravariant functor to topological groups.

$\Rep({\cal G})$ contains a subspace $\Irr({\cal G})$ of irreducible
representations.  Tensoring an irreducible representation with a
character yields a new irreducible representation, so $\Irr({\cal G})$
is a ${\cal G}^*$--space.  Because unitary representations have unique
decompositions into irreducible representations, we have the
following:

\begin{prop}
The underlying set of $\Rep({\cal G})$ is the free abelian monoid on
the set $\Irr({\cal G})$.
\end{prop}

\begin{cor}
\label{cor:finiterepspace}
If ${\cal G}$ is finite, $\Rep({\cal G})$ is the free abelian topological
monoid on the finite set $\Irr({\cal G})$.
\end{cor}

The following result, which allows us to examine the structure of
$\Rep({\cal G})$ explicitly for discrete nilpotent groups, first appeared
in \cite[Theorem~6.6]{lubotzkymagid85}.

\begin{thm}
\label{thm:irrspace}
Suppose ${\cal G}$ is discrete, finitely generated, and nilpotent.
There is a set $\{\rho_\alpha\}$ of representations of ${\cal G}$, all
factoring through finite quotients, and finite subgroups $I_\alpha <
{\cal G}^*$ such that
\[
\Irr({\cal G}) \cong \coprod_\alpha ({\cal G}^*/I_\alpha)
\cdot [\rho_\alpha]
\]
as a ${\cal G}^*$--space.  Additionally, for any $n \in \mb N$ there are
only finitely many $\rho_\alpha$ of dimension $n$.
\end{thm}

\begin{rmk}
\label{rmk:repspace}
The group ${\cal G}^*$ is a closed subgroup of $\prod \Uni(1)$, and
hence the spaces ${\cal G}^*/I_\alpha$ are finite unions of
(possibly zero--dimensional) torii.  Additionally, for any
$\rho_\alpha$ the subgroup $I_\alpha$ consists of those characters
$\psi$ such that $\psi \otimes_{\mb C} \rho_\alpha \cong \rho_\alpha$.  This
is necessarily finite: $\rho_\alpha$ factors through a finite
quotient ${\cal G}/N$, so any such $\psi$ must also factor through
${\cal G}/N$.  Therefore, $\psi$ can only range over the finite group of
characters of ${\cal G}/N$.
\end{rmk}

\begin{cor}
\label{cor:fiberprod}
If ${\cal G}$ is discrete, finitely generated, and nilpotent, there is
an isomorphism of ${\cal G}^*$--spaces
\[
\Irr({\cal G}) \cong {\cal G}^* \times_{({\compl{\cal G}})^*}
\Irr({\compl{\cal G}}),
\]
where $\compl{{\cal G}}$ is the profinite completion
\[
\lim_{[{\cal G}:N] < \infty} {\cal G}/N.
\]
\end{cor}

\begin{proof}
As in \fullref{thm:irrspace}, first write
\[
\Irr({\cal G}) \cong \coprod_\alpha {\cal G}^*/I_\alpha [\rho_\alpha],
\]
where the $\rho_\alpha$ factor through finite quotients of ${\cal
G}$.  Therefore, we can assume that the $\rho_\alpha$ are actually
elements of $\Irr({\compl{\cal G}})$.  If $\rho_\alpha$ factors through
a finite quotient of ${\cal G}$ and $\psi$ is a character of ${\cal
G}$ such that $\psi \otimes \rho_\alpha \cong \rho_\alpha$, $\psi$
itself must factor through a finite quotient of ${\cal G}$.
Therefore, the isotropy subgroup $I_\alpha$ of $\rho_\alpha$ is
actually a subgroup of $(\compl{\cal G})^*$.

Tensoring with characters of $\compl{{\cal G}}$ gives a map of
$({\compl{\cal G}})^*$--sets
\begin{equation}
\label{eq:mapequiv}
\coprod_\alpha ({\compl{\cal G}})^*/I_\alpha \cdot [\rho_\alpha] \to
\Irr({\compl{\cal G}}).
\end{equation}
We claim that this map is an isomorphism, which proves the corollary.
It is injective: composing with the map out to
$\Irr({\cal G})$ gives an injection
\[
\coprod (\compl{{\cal G}})^*/I_\alpha \to \coprod {\cal G}^*/I_\alpha.
\]
(The map $(\compl{{\cal G}})^* \to {\cal G}^*$ is injective, because a
character is determined by its values on the dense image of ${\cal
G}$.)  We now prove that the map of equation~\ref{eq:mapequiv} is
surjective.

Given an irreducible representation $\rho$ of ${\compl{\cal G}}$, its
restriction to ${\cal G}$ is an irreducible representation $\rho$ of
${\cal G}$ on a vector space $W$.  By the above decomposition of
$\Irr({\cal G})$, $\rho \cong \psi \otimes \rho_\alpha$ for some
$\alpha$: choose such an isomorphism $f \co \mb C \otimes
W_\alpha \to W$.  Consider the ${\cal G}$--vector space $\Hom_{\mb
C}(W_\alpha, W) \cong W_\alpha^* \otimes W$, where the dual and tensor
are both taken over $\mb C$.  ${\cal G}$ acts on $W$ and $W_\alpha$
through finite quotients, so the action of ${\cal G}$ on $\Hom_{\mb
C}(W_\alpha, W)$ also factors through a finite quotient of ${\cal
G}$.  This vector space contains a map $f'$,
defined by $f'(w) = f(1 \otimes w)$.  Then $g$ acts on $f'$ as follows:
\begin{eqnarray*}
(gf')(w) 
&=& g \cdot f'(g^{-1} w) \\
&=& g \cdot f(1 \otimes g^{-1} w) \\
&=& f(g \cdot (1 \otimes g^{-1} w)) \\
&=& f(\psi(g) \otimes w) \\
&=& \psi(g) f'(w).
\end{eqnarray*}
Therefore, the representation $\Hom_{\mb C}(W_\alpha, W$) contains a
1--dimensional subspace $\langle f' \rangle$ isomorphic to the
representation $\psi$, so $\psi$ factors through a finite quotient of
${\cal G}$.
\end{proof}

We are now in a position to prove the main result of this section.
For a set $X$, let $\Sym^k(X) = X^k/\Sigma_k$ be the $k$--fold
symmetric product of $X$, and for a based set $X$ let $\Sym^\infty(X)$
be the infinite symmetric product $\lim_k \Sym^k(X)$, where the limit is
taken by adding additional copies of the basepoint.

\begin{thm}
\label{thm:nilrep}
Suppose ${\cal G}$ is nilpotent, finitely topologically generated, and
either profinite or discrete.  Then $\Rep({\cal G})$, the space of
isomorphism classes of representations of ${\cal G}$, is homeomorphic
to $\Sym^\infty(\Irr({\cal G})_+)$, the free abelian topological
monoid generated by the space $\Irr({\cal G})$ of irreducible
representations of ${\cal G}$.
\end{thm}

\begin{proof}
If ${\cal G}$ is profinite, then from \fullref{exam:profinite} we
can make the identification $\Rep({\cal G}) = \colim_N \Rep({\cal G}/N)$
as $N$ ranges over open normal subgroups of ${\cal G}$.  The result
follows from \fullref{cor:finiterepspace}.

Now suppose that ${\cal G}$ is discrete.  We have a continuous map of
spaces $\Irr({\cal G}) \to \Rep({\cal G})$ that induces a continuous
map of monoids $\Sym^\infty(\Irr({\cal G})_+) \to \Rep({\cal G})$.
Because the underlying set of $\Sym^\infty(\Irr({\cal G})_+)$ is the
free abelian monoid on $\Irr({\cal G})$, this map is bijective.

We have also seen in \fullref{thm:irrspace} that $\Irr({\cal G})$
is compact in fibers over $\mb N$.  Because this is true for
$\Irr({\cal G})$, it is true for the symmetric products
$\Sym^k(\Irr({\cal G}))$, and hence true for the monoid
$\Sym^\infty(\Irr({\cal G})_+) \cong \coprod \Sym^k(\Irr({\cal G}))$.
The result follows because $\Rep({\cal G})$ is Hausdorff.
\end{proof}

\section{Deformation representation ring spectra}
\label{sec:repringspectra}
Let ${\cal G}$ be a topological group.  From this group ${\cal G}$, we
will now construct a deformation representation ring spectrum $R[{\cal
  G}]$ by analogy with the construction for a finite group.  The
constructions here are a modification of those that yield a
``deformation'' $K$--theory spectrum (see \cite{carlsson:kth}), and the
homotopy groups of $R[{\cal G}]$ will be shown in forthcoming work to
be involved in a spectral sequence that computes the homotopy groups
of deformation $K$--theory.

The space $\Rep({\cal G})$ of \fullref{sec:nilrep} is an abelian
topological monoid under $\oplus$, so we can apply an iterated
classifying space construction to it; the result is a (na\"ive pre-)
spectrum whose $k$'th space is $\{B^{(k)} \Rep({\cal G})\}$.  The
honestly commutative product structure induced by $\otimes$ on
$\Rep({\cal G})$ will give this object the structure of an
$E_\infty$--ring spectrum.

To make this construction more rigid, we will make use of the
$\Gamma$--spaces of Segal \cite{segal74}, using the smash product
defined by Lydakis in \cite{lydakis99}.  These are a model of
connective spectra particularly suited for this kind of application.
See Schwede \cite{schwede99} for an exposition of the homotopy theory
of these objects; in particular, the appendix contains an introduction
to the ``topological $\Gamma$--spaces'' that we will be using.  The
relationship of these objects with other structured categories of
spectra can be found in Mandell--May--Schwede--Shipley
\cite{mmss01}.  (In that article,
$\Gamma$--spaces are referred to as ${\cal F}$--spaces.)  The structured
ring and module objects created in this section can be converted into
symmetric spectra using a Quillen equivalence; by
\cite[Theorems~0.1,0.3]{mmss01}, the resulting objects are commutative
monoid objects in symmetric spectra.  Similarly, applying
\cite[Theorem~5.1]{schwede01}, the resulting monoids in symmetric
spectra can be converted using a Quillen equivalence to commutative
$\mb S$--algebras.

We first recall some definitions and properties of $\Gamma$--spaces.

\begin{defn}
$\Gamma^o$ is the category of finite based sets.
\end{defn}

\begin{defn}
A $\Gamma$--space is a functor $M \co \Gamma^o \to {\bf Top}_*$ from
finite based sets to based spaces such that $M(*) = *$.  A map of
$\Gamma$--spaces is a natural transformation of functors.
\end{defn}

\begin{defn}
The sphere spectrum $\Gamma$--space is defined by $\mb S(X) = X$.
\end{defn}

\begin{defn}
If $M$ is a $\Gamma$--space and $K$ is a pointed space, define the
tensor product $\Gamma$--space by
\[
(M \otimes K)(X) = M(X) \smsh{} K.
\]
\end{defn}

\begin{defn}
If $M$ and $N$ are $\Gamma$--spaces, the smash product $\Gamma$--space
is defined as follows: for $Z \in \Gamma^o$,
\[
M \smsh{} N(Z) = \colim_{X \smsh{} Y \to Z} M(X) \smsh{} N(Y)
\]
The object $\mb S \otimes K$ is the suspension spectrum
$\Gamma$--space associated to $K$.
\end{defn}

Left for exercises are proofs of the standard facts: the smash product
is a symmetric monoidal product on $\Gamma$--spaces with unit $\mb S$,
and there are identifications
\[
M \smsh{} (N \otimes K) \cong (M \smsh{} N) \otimes K.
\]

\begin{rmk}
This definition of the smash product is a left Kan extension of the
functor $M \smsh{} N\co \Gamma^o \times \Gamma^o \to {\bf Top}_*$
along the functor $\smsh{}\co \Gamma^o \times \Gamma^o \to \Gamma^o$.
Therefore, in order to define a map from $M \smsh{} N$ to $P$ it
suffices to define natural maps $M(X) \smsh{} N(Y) \to P(X \smsh{} Y)$.
\end{rmk}

A symmetric spectrum is constructed from a $\Gamma$--space $M$ in the
following way.  First, we prolong the functor $M$ to a functor on
finite pointed simplicial sets by applying $M$ levelwise and 
taking geometric realization.  Second, we note that there is a natural
assembly map as follows for $X \in \Gamma^o$:
\[
X \smsh{} M(Y) \cong \bigvee_X M(Y) \to M\left(\bigvee_X Y\right)
\cong M(X \smsh{} Y).
\]
This can also be prolonged to the case when $X$ is the geometric
realization of a finite simplicial set.  The resulting symmetric
spectrum is the sequence of spaces $M((S^1)^{\wedge n})$, with
structure maps coming from the symmetric group action on the smash
product and the structure map
\[
S^1 \smsh{} M\left((S^1)^{\wedge n}\right) \to M\left((S^1)^{\wedge
    (n+1)}\right).
\]
This association respects the symmetric monoidal structure.  We will
denote this spectrum by $\Sp(M)$.

\begin{defn}
If $A$ is an abelian topological monoid, define a $\Gamma$--space
$\tilde A$ as follows.  For $Z \in \Gamma^o$, define
\[
\tilde A(Z) = \Map_*(Z,A),
\]
where the unit of $A$ is regarded as its basepoint.  For $f \co Z \to
Z'$ and $\alpha \in \Map_*(Z,A)$, define
\[
(f_* \alpha)(z') = \sum_{f(z) = z'} \alpha(z).
\]
\end{defn}

\begin{prop}
A continuous bilinear pairing $\mu \co A \times B \to C$ of abelian
topological monoids naturally gives rise to a map $\tilde \mu \co
\tilde A \smsh{} \tilde B \to \tilde C$ of $\Gamma$--spaces.
\end{prop}

\begin{proof}
This is defined in the obvious way.  We define $\tilde \mu \co
\tilde A(X) \smsh{} \tilde B(Y) \to \tilde C(X \smsh{} Y)$ by
\[
\tilde \mu(\alpha \smsh{} \beta) (x \smsh{} y) = \mu(\alpha(x), \beta(y)).
\]
That this is natural in $X$ and $Y$ follows from the
bilinearity of $\mu$ as follows.  Suppose $f\co X \to X'$. Then
\begin{eqnarray*}
\tilde \mu(f_* \alpha \smsh{} \beta) (x' \smsh{} y)
&=& \mu(f_* \alpha(x'), \beta(y)) \\
&=& \mu\left(\sum_{f(x) = x'} \alpha(x), \beta(y)\right) \\
&=& \sum_{f(x)  = x'} \tilde \mu(\alpha \smsh{} \beta)(x\smsh{}y) \\
&=& (f \smsh{} 1)_* \tilde \mu(\alpha \smsh{} \beta)(x \smsh{} y).
\end{eqnarray*}
The naturality in $Y$ is symmetric.
\end{proof}

\begin{rmk}
Any abelian topological monoid $A$ has a unique bilinear pairing
$\mb N \times A \to A$ taking $1 \times a$ to $a$.  If $R$ is a
topological semiring (an object with addition and multiplication
operations, but no additive inverses), there is a unique map $\mb N
\to R$ respecting the multiplication.  As a result, in
$\Gamma$--spaces $\tilde{\mb N}$ is a commutative monoid, any $\tilde
A$ is a module over $\tilde{\mb N}$, and maps of topological monoids
realize to maps of $\tilde{\mb N}$--modules.  In other words, the
functor $\widetilde{(-)}$ has range in $\tilde{\mb N}$--modules.

Any $\tilde R$ comes equipped with a multiplicative map $\tilde{\mb N}
\to \tilde R$.  Additionally, if $A \times B \to C$ is a bilinear pairing,
the map $\tilde A \smsh{} \tilde B \to \tilde C$ respects the
$\tilde{\mb N}$--action.
\end{rmk}

We now define the deformation representation ring functor.

\begin{defn}
The deformation representation ring $R[{\cal G}]$ is the spectrum 
$\Sp(\widetilde{\Rep({\cal G})})$.
\end{defn}

For any abelian topological monoid $A$, the $\Gamma$--space $\tilde A$
is {\em special\/} (see \cite{schwede99}), meaning that the map
\[
\tilde A(X \vee Y) \to \tilde A(X) \times \tilde A(Y)
\]
is an equivalence for all $X, Y$.  (In this case, it is actually a
homeomorphism.)  As a result, the infinite loop space $\Omega^\infty
\Sp \tilde A$ can be analyzed.  The adjoint of the assembly map
$\tilde A(S^0) \to \Omega \tilde A(S^1)$ is the homotopy group
completion map $A \to \Omega B A$ of the monoid $A$.  Since $(\tilde A
\otimes S^n)(S^0)$ is connected, by \cite[Proposition~1.4]{segal74}
the map $\tilde A(S^n) \to \Omega \tilde A (S^n \smsh{} S^1)$ is a
homotopy equivalence for $n > 0$, 
implying that $\Omega^n \tilde A(S^n) \simeq \Omega B A$ for $n > 0$.
Taking limits shows that $\Omega^\infty \Sp(\tilde A) \simeq \Omega B A$.

$R[-]$ is a contravariant functor from groups to commutative monoids
in symmetric spectra.  In fact, $\Omega^\infty \Sp(\tilde{\mb N})
\simeq \mb Z$, so the spectrum $\Sp(\tilde{\mb N})$ is an
Eilenberg--MacLane spectrum equivalent to $\eilm{\mb Z}$.  (We will
show this explicitly in \fullref{prop:repdga}.)  The functor
$R[-]$ takes values in $\eilm{\mb Z}$--algebras. There is also a
natural augmentation $\varepsilon \co R[{\cal G}] \to \eilm{\mb
Z}$ coming from the augmentation $\Rep({\cal G}) \to \mb N$.

\begin{prop}
\label{prop:repdga}
Suppose ${\cal G}$ is finitely topologically generated, nilpotent, and
either discrete or profinite.  Then $R[{\cal G}]$ is equivalent
to $\eilm{\mb Z} \smsh{} \Sigma^\infty_+ \Irr({\cal G})$ as a module
over the group ring spectrum $\eilm{\mb Z} \smsh{} \Sigma^\infty_+
{\cal G}^*$.

As a result, $\pi_*(R[{\cal G}]) \cong H_*(\Irr({\cal G}))$.
\end{prop}

\begin{proof}
We have a natural identification of spaces $\widetilde{\Rep({\cal
G})}(S^0) = \Rep({\cal G})$, and so there are natural maps of pointed
spaces:
\[
({\cal G}^*)_+ \to \Irr({\cal G})_+ \to R[{\cal G}](S^0).
\]
For $X \in \Gamma^o$, taking a wedge over the points of $X$ gives
natural maps as follows:
\[
X \smsh{} ({\cal G}^*)_+ \to X \smsh{} \Irr({\cal G})_+ \to X \smsh{}
\widetilde{\Rep({\cal G})}(S^0) \to \widetilde{\Rep({\cal G})}(X).
\]
The right-hand map is the assembly map.

Naturality implies that these give rise to maps 
\[
\mb S \otimes ({\cal G}^*)_+ \to \mb S \otimes \Irr({\cal G})_+ \to
\widetilde{\Rep({\cal G})}
\]
of $\Gamma$--spaces.  Additionally, for $X,Y \in \Gamma^o$ there
is a commutative diagram as follows. 
$$\xymatrix{ 
X \smsh{} ({\cal G}^*)_+ \smsh{} Y \smsh{} ({\cal G}^*)_+ \ar[r]
\ar[d]_\otimes &
{\widetilde{\Rep({\cal G})}}(X) \smsh{} {\widetilde{\Rep({\cal
G})}}(Y) \ar[d]^\otimes \\ 
X \smsh{} Y \smsh{} ({\cal G}^*)_+ \ar[r] &
\widetilde{\Rep({\cal G})}(X \smsh{} Y) 
}$$
This is natural in $X$ and $Y$, and as a result it translates into the
following diagram of $\Gamma$--spaces.
$$\xymatrix{
(\mb S \otimes ({\cal G}^*)_+) \smsh{} (\mb S \otimes ({\cal G}^*)_+)
\ar[r] \ar[d] &
\widetilde{\Rep({\cal G})} \smsh{} \widetilde{\Rep({\cal G})} \ar[d] \\
\mb S \otimes ({\cal G}^*)_+  \ar[r] &
\widetilde{\Rep({\cal G})}
}$$
This gives $\widetilde{\Rep({\cal G})}$ a natural structure of an
algebra over the group ring $\Gamma$--space $\mb S[{\cal G}^*] = \mb S
\otimes ({\cal G}^*)_+$.  $\widetilde{\Rep({\cal G})}$ is already a
commutative $\tilde{\mb N}$--algebra, so $\widetilde{\Rep({\cal G})}$
is naturally an algebra over $\tilde{\mb N} \smsh{} \mb S[{\cal G}^*]
\cong \tilde{\mb N}[{\cal G}^*]$. 

The map $\mb S \otimes \Irr({\cal G})_+ \to \widetilde{\Rep({\cal
G})}$ also extends naturally to a map $\tilde{\mb N} \otimes
\Irr({\cal G})_+ \to \widetilde{\Rep({\cal G})}$.  Similarly, we can
give $\tilde{\mb N} \otimes \Irr({\cal G})_+$ the structure of a
module over $\tilde{\mb N}[{\cal G}^*]$ in a natural way making the
map $\tilde{\mb N} \otimes \Irr({\cal G})_+ \to \widetilde{\Rep({\cal
G})}$ a map of modules over $\tilde{\mb N}[{\cal G}^*]$.  We claim
that this map gives a stable equivalence of spectra.  In fact, this
follows from \fullref{thm:nilrep} and a generalized Dold--Thom
theorem, which states that the spectrum
$\Sp\left(\widetilde{\Sym^\infty(M)}\right)$ is equivalent to $\eilm{\mb Z}
\smsh{} \Sigma^\infty_+ M$ for any space $M$ of the homotopy type of a
CW--complex.  The natural map inducing the Dold--Thom equivalence
\[
\tilde{\mb N} \otimes \Irr({\cal G})_+ \to
\widetilde{\Sym^\infty(\Irr({\cal G}))},
\]
upon applying $\Sp$, becomes an equivalence
\[
\eilm{\mb Z} \smsh{} \Sigma^\infty_+ \Irr({\cal G})_+ \to R[{\cal G}].
\]
Since this was a map of modules over the $\Gamma$--space $\tilde{\mb
N}[{\cal G}^*]$, applying the monoidal functor $\Sp$ gives a map of
modules over the group ring spectrum $\eilm{\mb Z} \smsh{}
\Sigma^\infty_+ {\cal G}^*$.
\end{proof}

\begin{cor}
\label{cor:finiterepring}
When ${\cal G}$ is finite, $R[{\cal G}]$ is an Eilenberg--MacLane
spectrum associated to the ordinary complex representation ring of
${\cal G}$, and when ${\cal G}$ is profinite, $R[{\cal G}] \cong
\colim_{[{\cal G}:N] < \infty} R[{\cal G}/N]$ is also an
Eilenberg--MacLane spectrum.
\end{cor}

\begin{proof}
In these cases, $\Irr({\cal G})$ is discrete.
\end{proof}

\section{$R[{\cal G}]$ and Bousfield localization}
\label{sec:compiso}

Throughout this section ${\cal G}$ is a finitely generated discrete
nilpotent group, having profinite completion $\compl{{\cal G}}$ and
pro-$p$--completions $\compl{{\cal G}}_p$.  Let $\ell$ denote a prime at
which completions of spectra will be taken.  The ``localization'' of
an $R$--module $M$ will always refer to Bousfield localization with
respect to $\eilm{\mb F_\ell}$, where $\eilm{\mb F_\ell}$ has an
explicit module structure over some base $\mb S$--algebra $R$.  (In
this section, we will make use of $\mb S$--modules as our base category
of spectra, unless otherwise specified.  As described in the beginning
of \fullref{sec:repringspectra}, the ring objects of the previous
section will be implicitly converted into $\mb S$--algebras, retaining
their names.)

There is an augmentation map $\varepsilon \co R[{\cal G}] \to \eilm{\mb
F_\ell}$ given by taking a representation to its dimension mod
$\ell$.  This is a map of commutative $\mb S$--algebras.

In this section we will proceed with the proof of
\fullref{thm:main} by breaking it up into the two following
statements.

\begin{prop}
\label{prop:main2}
The map $R[\compl{{\cal G}}] \to R[{\cal G}]$ is a
multiplicative $\eilm{\mb F_\ell}$--equivalence, and 
the map $R[\compl{{\cal G}}_\ell] \to R[\compl{{\cal G}}]$ is a
multiplicative $\eilm{\mb F_\ell}$--equivalence.
\end{prop}

\begin{proof}
The proof of the first statement is organized by progressively proving
it for more general cases.

The case when ${\cal G}$ is a finite abelian group follows because
$\compl{{\cal G}} \cong {\cal G}$.  \fullref{sec:groupZ} carries
out the proof in the special case ${\cal G} = \mb Z$.

In \fullref{sec:abcase}, we prove \fullref{lem:groupprod}, which
shows that if $R[G] \to R[G']$ is a multiplicative $\eilm{\mb
F_\ell}$--equivalence, so is $R[G \times H] \to R[G' \times H]$.  If
${\cal G}$ is a finitely generated abelian group,
\fullref{lem:groupprod} then allows us to prove the theorem by using
the existence of a decomposition ${\cal G} \cong \mb Z^d \oplus A$,
where $A$ is a finite abelian group.

In \fullref{sec:generalnil}, we finish the proof of the first
statement by making use of \fullref{prop:compiso} and the
structure of $R[{\cal G}]$ and $R[\compl{{\cal G}}]$ as algebras over
$R[({\compl{{\cal G}}})_{\rm ab}]$.

If ${\cal G}$ is nilpotent, then by definition some iterated bracket map
\[
[-,[-,[\cdots]]]\co {\cal G}^r \to {\cal G}
\]
is trivial.  The image of ${\cal G}^r$ is dense in $(\compl{{\cal
G}})^r$, and so the iterated bracket map is also trivial for
$\compl{\cal G}$.  Hence $\compl{\cal G}$ is nilpotent.

Using this, the proof of the second statement is the content of
\fullref{sec:nonl}, \fullref{cor:nonl}.
\end{proof}

\subsection{The case of ${\cal G} = \mb Z$}
\label{sec:groupZ}
First, we compute the representation rings of the groups $\mb Z$, its
profinite completion $\hat{\mb Z}$, and the $p$--adic integers $\mb
Z_p$, where $p$ is a prime.  By the results in
\fullref{sec:repringspectra}, it suffices to compute the spaces of
irreducible representations.

Let $\mu_\infty$ be the group of roots of unity in $\mb C$, and let
$\mu_{p^\infty}$ be the subgroup consisting of elements whose order is
a power of $p$.

\begin{lem}
The maps $\Irr(\mb Z_p) \to \Irr(\hat{\mb Z}) \to \Irr (\mb Z)$ are
the maps
\[
\mu_{p^\infty} \to \mu_\infty \to S^1, 
\]
where $\mu_\infty$ and $\mu_{p^\infty}$ have the discrete topology.
\end{lem}

\begin{proof}
Any unitary representation of an abelian group is diagonalizable,
so the spaces of irreducible representations must be precisely the
spaces of 1--dimensional representations.
\end{proof}

\begin{prop}
The map $R[\hat{\mb Z}] \to R[\mb Z]$ is a multiplicative $\eilm{\mb
F_\ell}$--equivalence.
\end{prop}

\begin{proof}
We need to show that the following map is an equivalence:
\[
\eilm{\mb F_\ell} \to \eilm{\mb F_\ell} \smsh{R[\hat{\mb Z}]}
R[\mb Z].
\]
Choose a sequence of primitive roots of unity $\zeta_m$ of $\mu_m$
such that $\zeta_m = (\zeta_{dm})^d$.  We identify $R[\mb Z]$ with the
group ring spectrum $\eilm{\mb Z}[S^1]$.

By \cite[Proposition~IV.7.5]{ekmm97}, for a $\mb S$--algebra $R$ with
right module $M$ and left module $N$, there is an equivalence of
the geometric realization of the bar complex $B(M,R,N)$ with $M
\smsh{R} N$.  If $M$ and $N$ are modules over a directed system of
algebras $\{R_\alpha\}$ with homotopy colimit equivalent to $R$, taking
homotopy colimits implies that
\[
M \smsh{R} N \simeq \hocolim M \smsh{R_\alpha} N.
\]
(If the system is not directed, the homotopy colimit as $\mb
S$--algebras does not necessarily coincide with the homotopy colimit as
$\mb S$--modules.)  Therefore, the smash product commutes with directed
homotopy colimits in the $\mb S$--algebra variable.

The map $\hocolim_m \eilm{\mb Z}[\mu_m] \to R[\hat{\mb 
Z}]$ is an equivalence.  It therefore suffices to show that the
map
\[
\eilm{\mb F_\ell} \to \hocolim \eilm{\mb F_\ell} \smsh{\eilm{\mb Z}
  [\mu_m]} \eilm{\mb Z}[S^1]
\]
is an equivalence.

As a $\mu_m$--space, $S^1$ is homeomorphic to the homotopy colimit of
the diagram
\[
\mu_m \mathop{\rightrightarrows}^{\zeta_m}_{1} \mu_m.
\]
Applying $\eilm{\mb Z} \smsh{} \Sigma^\infty_+$, this leads to the
exact triangle of $\eilm{\mb Z}[\mu_m]$--modules
\[
\eilm{\mb Z}[\mu_m] \longoverto^{\zeta_m - 1} \eilm{\mb Z}[\mu_m] \to
\eilm{\mb Z}[S^1].
\]
Recall that the augmentation $\mb Z[\mu_m] \to \mb F_\ell$ is the
unique ring map taking the elements of $\mu_m$ to $1$.  Applying
$\eilm{\mb F_\ell} \smsh{\eilm{\mb Z}[\mu_m]} -$ to the above exact
triangle gives the exact triangle
\[
\eilm{\mb F_\ell} \overto^0 \eilm{\mb F_\ell} \to 
\eilm{\mb F_\ell} \smsh{\eilm{\mb Z}[\mu_m]} \eilm{\mb Z}[S^1].
\]
The inclusion $\iota\co \mu_m \to \mu_{dm}$ induces the following map
of exact triangles.
$$\xymatrix{
\eilm{\mb Z}[\mu_m] \ar[r]^{\zeta_m - 1} \ar[d]_{({\sum_{i=1}^d
  \zeta_{dm}^i)}\cdot \iota} &
\eilm{\mb Z}[\mu_m] \ar[r] \ar[d]_\iota &
\eilm{\mb Z}[S^1] \ar[d]_1\\
\eilm{\mb Z}[\mu_{dm}] \ar[r]^{\zeta_{dm} - 1} &
\eilm{\mb Z}[\mu_{dm}] \ar[r]&
\eilm{\mb Z}[S^1]
}$$
As a result, after smashing with $\eilm{\mb F_\ell}$, there is the
following induced map of exact triangles.
$$\xymatrix{
\eilm{\mb F_\ell} \ar[r]^{0} \ar[d]_{d} & 
\eilm{\mb F_\ell} \ar[r] \ar[d]_{1} &
\eilm{\mb F_\ell} \smsh{\eilm{\mb Z}[\mu_m]} \eilm{\mb Z}[S^1] \ar[d]
\\ 
\eilm{\mb F_\ell} \ar[r]^{0} &
\eilm{\mb F_\ell} \ar[r] &
\eilm{\mb F_\ell} \smsh{\eilm{\mb Z}[\mu_{dm}]} \eilm{\mb Z}[S^1]
}$$
Taking homotopy colimits in $m$, we find that the map
\[
\eilm{\mb F_\ell} \to \hocolim \eilm{\mb F_\ell} \smsh{\eilm{\mb
    Z}[\mu_{m}]} \eilm{\mb Z}[S^1]
\]
is an equivalence, as desired.
\end{proof}

\subsection{The case of ${\cal G}$ finitely generated abelian}
\label{sec:abcase}

If ${\cal G}$ is a finitely generated abelian group, we will show in this
section that the map $R[\compl{{\cal G}}] \to R[{\cal G}]$ induces an
isomorphism after localization.

If ${\cal G}$ is cyclic, then either ${\cal G}$ is finite or ${\cal G}
= \mb Z$, and both of these cases have already been proven.  The case
of a finitely generated abelian group follows immediately from the
following lemma.

\begin{lem}
\label{lem:groupprod}
If $G$, $G'$, and $H$ are nilpotent groups and a map $G \to G'$
induces a multiplicative $\eilm{\mb F_\ell}$--equivalence $R[G']
\to R[G]$, then so does the map $G \times H \to G' \times H$.
\end{lem}

\begin{proof}
We first note that from \fullref{thm:nilrep}, $R[G \times H]
\cong R[G] \smsh{\eilm{\mb Z}} R[H]$ because $\Irr(G \times H) \cong
\Irr(G) \times \Irr(H)$.

We need to show that the map $R[G \times H] \to R[G' \times
H]$ induces an isomorphism after applying $\eilm{\mb F_\ell}
\smsh{R[G\times H]} (-)$.

As $R[G] \smsh{\eilm{\mb Z}} R[H]$--modules, we have that
\[
\eilm{\mb F_\ell} \simeq \eilm{\mb F_\ell} \smsh{\eilm{\mb Z}}
\eilm{\mb Z}\]
and
\[
R[G' \times H] \simeq R[G'] \smsh{\eilm{\mb Z}} R[H].
\]
By carrying out the proof of \cite[Proposition~III.3.10]{ekmm97}, with
base ring $\eilm{\mb Z}$ rather than $\mb S$, we get a smash product identity
\[
\left(A \smsh{\eilm{\mb Z}} B\right) \smsh{(R \smsh{\eilm{\mb Z}} R')}
\left(C \smsh{\eilm{\mb Z}} D\right) \cong \left(A \smsh{R} C\right)
\smsh{\eilm{\mb Z}} \left(B \smsh{R'} D\right).
\]
This allows us to obtain the following equivalences:
$$
\eqalignbot{
\eilm{\mb F_\ell} \smsh{R[G \times H]} R[G'\times H] &\simeq
\left(\eilm{\mb F_\ell}\smsh{R[G]}R[G']\right) \smsh{\eilm{\mb Z}}
\left(\eilm{\mb Z} \smsh{R[H]} R[H]\right) \cr
&\simeq \left(\eilm{\mb F_\ell} \smsh{R[G]}R[G]\right) \smsh{\eilm{\mb Z}}
\left(\eilm{\mb Z}\smsh{R[H]}R[H]\right) \cr
&\simeq \eilm{\mb F_\ell}\smsh{R[G\times H]} \left(R[G] \smsh{\eilm{\mb Z}}
  R[H]\right) \cr
&\simeq \eilm{\mb F_\ell}.}
\proved$$
\end{proof}

\subsection{The case of a general nilpotent ${\cal G}$}
\label{sec:generalnil}

We now complete the proof of the first statement of
\fullref{prop:main2}.  The general case of a discrete finitely
generated nilpotent group ${\cal G}$ will follow from the machinery of
\fullref{sec:completions} and our explicit description of the
space of irreducible representations of a nilpotent group.

Let $k$ be $R[(\compl{{\cal G}})^*]$, which is an Eilenberg--MacLane
spectrum associated to the representation ring of $({\compl{\cal
G}})_{\rm ab}$ by \fullref{cor:finiterepring}.  The restriction
map $k \to R[\compl{{\cal G}}]$ gives $R[\compl{{\cal G}}]$ and $R[{\cal
G}]$ the structure of $k$--algebras.

By \fullref{prop:compiso}, it suffices to show that there is
an equivalence
\[
\eilm{\mb F_\ell} \smsh{k} R[\compl{{\cal G}}] \to \eilm{\mb F_\ell}
  \smsh{k} R[{\cal G}].
\]
From \fullref{cor:fiberprod} we have that
\[
\Irr({\cal G}) \cong {\cal G}^* \mathop{\times}_{(\compl{{\cal G}})^*}
\Irr(\compl{{\cal G}}) .
\]
By \fullref{thm:irrspace}, the space $\Irr(\compl{{\cal G}})$
breaks up as a coproduct of orbits of $(\compl{{\cal G}})^*$ with
finite stabilizers.  If $K$ is a finite subgroup of $(\compl{{\cal
G}})^*$, then $K$ acting freely on ${\cal G}^*$ implies that there is
a weak equivalence of ${\cal G}^*/K$ with the realization of the bar
construction $B\left({\cal G}^*,K,*\right)$.  The suspension spectrum
functor is strong symmetric monoidal by
\cite[Proposition~II.1.2]{ekmm97}, so applying $\Sigma^\infty_+$ gives
an equivalence
\begin{eqnarray*}
\Sigma^\infty_+ ({\cal G}^*/K) &\simeq&
B\left(\Sigma^\infty_+{\cal G}^*,\Sigma^\infty_+ K,\mb S\right)\\
&\simeq& 
\Sigma^\infty_+{\cal G}^* \smsh{\Sigma^\infty_+ K} \mb S.
\end{eqnarray*}
Similarly,
\[
\Sigma^\infty_+ (\compl{{\cal G}})^*/K \simeq
\Sigma^\infty_+(\compl{{\cal G}})^* \smsh{\Sigma^\infty_+ K} \mb S.
\]
The map $\Sigma^\infty_+(\compl{{\cal G}})^*/K \to \Sigma^\infty_+{\cal
G}^*/K$ 
then induces an equivalence
\[
\Sigma^\infty_+ {\cal G}^* \smsh{\Sigma^\infty_+ (\compl{{\cal G}})^*}
\Sigma^\infty_+(\compl{{\cal G}})^*/K \to \Sigma^\infty_+{\cal G}^*/K.
\]
Smashing with $\eilm{\mb Z}$ and taking the union over all orbits, we
find that
\[
\eilm{\mb Z} \smsh{} \Sigma^\infty_+ \Irr({\cal G}) \simeq
\left(\eilm{\mb Z} \smsh{} \Sigma^\infty_+ {\cal G}^*\right) \smsh{k}
\left(\eilm{\mb Z} \smsh{} \Sigma^\infty_+ \Irr(\compl{{\cal G}})\right).
\]
Applying $\eilm{\mb F_\ell} \smsh{k} -$ to this equivalence gives the
following new equivalences.
\begin{eqnarray*}
\eilm{\mb F_\ell} \smsh{k} R[{\cal G}] &\simeq& 
\eilm{\mb F_\ell} \smsh{k} 
\left(\eilm{\mb Z} \smsh{} \Sigma^\infty_+ \Irr({\cal G})\right) \\ 
&\simeq& \eilm{\mb F_\ell} \smsh{k} 
\left(\eilm{\mb Z} \smsh{} \Sigma^\infty_+ {\cal G}^*\right) \smsh{k}
\left(\eilm{\mb Z} \smsh{} \Sigma^\infty_+ \Irr(\compl{{\cal G}})\right)\\
&\simeq& \left(\eilm{\mb F_\ell} \smsh{k} 
R[{\cal G}_{\rm ab}] \right)\smsh{k}
\left(\eilm{\mb Z} \smsh{} \Sigma^\infty_+ \Irr(\compl{{\cal G}})\right)\\
&\simeq& \eilm{\mb F_\ell} \smsh{k} R[{\compl{\cal G}}]
\end{eqnarray*}
(The equivalence $\eilm{\mb F_\ell} \to \eilm{\mb F_\ell} \smsh{k}
R[{\cal G}_{\rm ab}]$ follows from the results of
\fullref{sec:abcase}, applied to ${\cal G}_{\rm ab}$.)

\subsection{Localizations for non--$\ell$--groups}
\label{sec:nonl}

The result of this section is that the representation rings of
profinite groups with no $\ell$--primary part become trivial upon
localization.  The first result of this section does not depend on
${\cal G}$ being nilpotent.

\begin{lem}
Suppose ${\cal G}$ is a profinite group whose finite quotients all have order
prime to $\ell$.  Then the augmentation map $R[{\cal G}] \to \eilm{\mb Z}$
is a multiplicative \eilm{\mb F_\ell}--equivalence.
\end{lem}

\begin{proof}
Note that $R[{\cal G}] \simeq \hocolim R[{\cal G}_\beta]$, where the
${\cal G}_\beta$ form the directed system of finite quotients of
${\cal G}$.  As a result, it is therefore necessary and sufficient
to prove that the map $\eilm{\mb F_\ell} \to \eilm{\mb F_\ell}
\smsh{R[{\cal G}]} \eilm{\mb Z}$ is an equivalence when ${\cal G}$
is a finite prime-to--$\ell$--group. Because $R[{\cal G}] \simeq
\eilm{R}$ for a ring $R$, the derived category of $R$--modules, with
derived tensor product, is equivalent to the category of modules
over $R[{\cal G}]$, by \cite[Theorem~IV.2.4]{ekmm97}.  Therefore, it
suffices to show that the map $R \to \mb Z$ induces an isomorphism
\[
\mb F_\ell \to \mb F_\ell \otimes^{\mb L}_{R} \mb Z.
\]
$R$ is projective over $\mb Z$, so we can compute these $\Tor$
groups using the bar complex
\[
\cdots \to \mb F_\ell \otimes R \otimes R \otimes \mb Z \to \mb
F_\ell \otimes R \otimes \mb Z \to \mb F_\ell \otimes \mb Z \to 0,
\]
where the tensor products are over $\mb Z$.  However, this is
obviously isomorphic to another bar complex,
\[
\cdots \to \mb F_\ell \otimes (\mb F_\ell \otimes R) \otimes (\mb
F_\ell \otimes R) \otimes \mb F_\ell \to \mb F_\ell \otimes (\mb
F_\ell \otimes R) \otimes \mb F_\ell \to \mb
F_\ell \otimes \mb F_\ell \to 0,
\]
which is the bar complex for computing $\Tor^{\mb F_\ell \otimes
  R}_*(\mb F_\ell, \mb F_\ell)$.

If $\bar{\mb F}_\ell$ is an algebraic closure of $\mb F_\ell$, then
$\bar{\mb F}_\ell \otimes -$ is an exact functor from $\mb F_\ell
\otimes R$--modules to $\bar{\mb F}_\ell \otimes R$--modules that
takes projective objects to projective objects and commutes with
tensor products.  Therefore,
\[
\bar{\mb F}_\ell \otimes \Tor_*^{\mb F_\ell \otimes R} ( \mb
F_\ell, \mb F_\ell) \cong \Tor_*^{\bar{\mb F}_\ell \otimes R}
  (\bar{\mb F}_\ell, \bar{\mb F}_\ell).
\]
The characteristic $\ell$ does not divide $|{\cal G}|$, so $\bar{\mb
F}_\ell \otimes R$ is isomorphic, by basic character theory, to the
ring of $\bar{\mb F}_\ell$--valued class functions on ${\cal G}$.  In
other words, $\bar{\mb F}_\ell \otimes R \cong \prod \bar{\mb
F}_\ell$.  This ring is semisimple, and so it has no higher $\Tor$
groups.

Because $\bar{\mb F}_\ell \otimes (-)$ is faithful on $\mb
F_\ell$--modules, this shows that
\[
0 \cong \Tor_*^{\mb F_\ell \otimes R}(\mb F_\ell, \mb F_\ell) \cong
\Tor_*^{R}(\mb F_\ell, \mb Z)
\]
for $* > 0$, as desired.
\end{proof}

\begin{cor}
\label{cor:nonl}
If $\cal G$ is a nilpotent profinite group, the map
$R[\compl{{\cal G}}_\ell] \to R[{\cal G}]$ is a multiplicative
\eilm{\mb F_\ell}-equivalence.
\end{cor}

\begin{proof}
Because $\cal G$ is nilpotent and profinite, it is an inverse limit
of nilpotent finite groups, each of which is canonically a product
of its (unique) Sylow subgroups \cite[Exercise~8.8]{serre77}.  This
makes ${\cal G}$ an inverse limit of groups of the form $K_1 \times
K_2$, where $K_1$ ranges over the $\ell$--group quotients of ${\cal G}$
and $K_2$ ranges over the prime-to--$\ell$ quotients.  Therefore, ${\cal
G}$ can be canonically expressed as a direct product ${\cal G}
\cong \compl{{\cal G}}_\ell \times H$, where $H$ is the inverse limit
of groups whose orders are prime to $\ell$.  The result now follows
from \fullref{lem:groupprod}.
\end{proof}

\section{Localization in the case of the Heisenberg group}
\label{sec:heiscomp}

Let $H$ be the integer Heisenberg group; ie, the group of $3 \times
3$ upper triangular matrices of the form:
\[
\begin{bmatrix}
1 & \mb Z & \mb Z \\ 0 & 1 & \mb Z \\ 0 & 0 & 1
\end{bmatrix}
\]
Alternatively, $H$ has presentation
\[
\left\langle x, y, z\ \big|\  [x,z] = [y,z] = 1,
  [x,y] = z\right\rangle.
\]

\subsection{The representation ring of the Heisenberg group}

The structure of the set of irreducible representations of the group
$H$ was made explicit in \cite{nm89}.  The main result of this section
is to compute the product structure in the representation ring.
Recall that $\mu_\infty$ is the set of roots of unity in $\mb C$, and
$\mu_{p^\infty}$ is the set of roots of $p$--power order.

\begin{prop}
\label{prop:Hrepring}
The space $H^*$of irreducible unitary representations of $H_{ab}$, the
abelianization of $H$, is the torus $S^1 \times S^1$.  As a space acted
on by $H^*$ by tensor product,
\[
\Irr(H) \cong \coprod_{\zeta \in \mu_{\infty}} \left( S^1 \times S^1 /
  I_\zeta \right) [V_\zeta],
\]
where $I_\zeta$ is the subgroup generated by $(\zeta,1)$ and
$(1,\zeta)$.  The $[V_\zeta]$ are a specific family of orbit
representatives, with $[V_1]$ the trivial representation.  The tensor
product on $\Irr(H)$ is given as follows.

If $\zeta$ and $\zeta'$ have relatively prime orders,
$[V_\zeta][V_{\zeta'}] = [V_{\zeta\zeta'}]$. Otherwise, suppose
$\zeta$ has order $p^k$, $\zeta'$ has order $p^l$, and $\zeta\zeta'$
has order $p^m$, with $k \geq l$.  Then the product is as follows.
\[
[V_\zeta] [V_{\zeta'}] \cong \left\{
\begin{matrix}
p^l [V_{\zeta\zeta'}] & \hbox{ if } k = m \\
p^m \left[\sum_{i,j = 1}^{p^{k-m}} (\zeta^i,
  \zeta^j)\right] [V_{\zeta\zeta'}] & \hbox{ if } k > m
\end{matrix}
\right.
\]
\end{prop}

\begin{proof}
We will first recall the isomorphism classification of irreducible
unitary representations of $H$.  In \cite[Theorem~4]{nm89}, it is
proved that simple representations of $H$ of degree $r$ are
classified by triples $(\alpha, \beta, \zeta)$ of nonzero complex
numbers, with $\zeta$ a primitive $r$'th root of unity, such that
$(\alpha, \beta, \zeta) \sim (\alpha', \beta', \zeta')$ if and only
if $\alpha^r = (\alpha')^r$, $\zeta = \zeta'$, and $(\beta)^r =
(\beta')^r$.  (The triple $(\alpha,\beta,\zeta)$ corresponds to
$[\alpha^r,\beta,\zeta]$ in the notation of \cite{nm89}; the
notation we use makes multiplicative relations simpler.)

More specifically, the representation $(\alpha,\beta,\zeta)$, for
$\zeta$ a primitive $r$'th root of unity, is induced from the closed
subgroup $H_r$ generated by $\langle x^r, y, z \rangle$.  Because
$[x^r,y] = z^{r}$ in $H$, we can define a character of $H_r$ by
\[
\omega\co x^r \mapsto \alpha^r,\ y \mapsto \beta,\ z \mapsto \zeta.
\]
Then $(\alpha,\beta,\zeta)$ corresponds to the representation of $H$
induced from $\omega$.  This representation is characterized by the
properties that on this representation, $z$ acts by the scalar $\zeta$,
the action of $x$ has an eigenvalue of $\alpha$, and the action of $y$
has an eigenvalue of $\beta$.

Specifically, the group of characters of $H$ is the collection of
representations $(\alpha, \beta, 1)$ for $\alpha, \beta \in \mb
C^\times \times \mb C^\times$.  In this representation, $x$ acts by
$\alpha$ and $y$ acts by $\beta$.  This character admits a unitary
structure if and only if $(\alpha,\beta) \in S^1 \times S^1$.

This characterization of irreducible representations allows us to
determine the multiplicative structure in the representation ring.  If
$g$ acts on $V$ with eigenvalues $\{\alpha_i\}_i$ and on $W$ with
eigenvalues $\{\alpha'_j\}_j$, then it acts on $V \otimes W$ with
eigenvalues $\{\alpha_i \alpha'_j\}_{i,j}$.

For example, under tensor product we have the formula
\[
(\alpha,\beta,\zeta) \otimes (\alpha',\beta',1) \cong (\alpha\alpha',
\beta\beta', \zeta),
\]
because in the tensor product representation the element $z$ still acts
by the scalar $\zeta$ and the elements $x$ and $y$ have eigenvectors
with eigenvalues $\alpha\alpha'$ and $\beta\beta'$, respectively.

By construction, if $\alpha$ and $\beta$ lie on the unit circle, the
character $\omega$ admits a unitary structure, and hence so does the
induced representation $(\alpha,\beta,\zeta)$.  Conversely, $\alpha$
and $\beta$ appear as eigenvalues in $(\alpha,\beta,\zeta)$, so this
representation admits a unitary structure only if $\alpha$ and $\beta$
lie on the unit circle.

The above considerations give the following decomposition of
$\Irr(H)$ as an $H^*$--space.  Explicity, 
\[
\Irr(H) \cong \coprod_{\zeta \in \mu_{\infty}} \left(S^1 \times S^1
  \Big/ \langle \zeta \rangle \times \langle
  \zeta \rangle\right) \cdot (1,1,\zeta).
\]
The homotopy groups of $R[H]$ are the homology groups of this space.
In order to determine the ring structure, it remains to identify the
tensor product $(1,1,\zeta) \otimes (1,1,\zeta')$.

Note that $z$ acts by the scalar $\zeta\zeta'$ on $(1,1,\zeta) \otimes
(1,1,\zeta')$, and so this tensor product is a direct sum of
representations of the form $(\alpha,\beta,\zeta\zeta')$.  If $\zeta$
and $\zeta'$ have relatively prime orders, $(1,1,\zeta) \otimes
(1,1,\zeta')$ has the same dimension as $(1,1,\zeta\zeta')$ and
contains eigenvectors of eigenvalue 1 for both $x$ and $y$, so
$(1,1,\zeta\zeta') = (1,1,\zeta) \otimes (1,1,\zeta')$.

Let $r$ and $s$ be the orders of $\zeta$ and $\zeta'$ respectively.
Let $t$ be the order of $\zeta \zeta'$, and $\eta$ be a primitive
generator for the group $\langle \zeta, \zeta' \rangle$ with order
$dt$ such that $\eta^d = \zeta\zeta'$.

The eigenvalues of $x$ and $y$ in $(1,1,\zeta)$ are
$\{\zeta^i\}_{i=1}^{r}$, and similarly for $(1,1,\zeta')$.
Therefore, we have that
\[
(1,1,\zeta) \otimes (1,1,\zeta') = \sum_{i,j=1}^{d} a_{i,j}
\left(\eta^i, \eta^j, \zeta\zeta'\right)
\]
for some coefficients $a_{i,j}$.

The representation $(1,1,\zeta)$ is invariant under tensoring with the
characters $(\zeta,1,1)$ and $(1,\zeta,1)$, so the same is true for is
$(1,1,\zeta) \otimes (1,1,\zeta')$, and similarly for multiplying by
$(\zeta',1,1)$ and $(1,\zeta',1)$.  This forces the coefficients
$a_{i,j}$ to be constant in $i$ and $j$.

As a result, calculating the dimensions of both sides, we find that
\[
(1,1,\zeta) \otimes (1,1,\zeta') = \frac{rs}{td^2}\sum_{i,j=1}^d
(\eta^i,\eta^j,\zeta\zeta').
\]
The formulation of the proposition can be recovered by writing
$[V_\zeta]$ for $(1,1,\zeta)$ and examining the following cases.

If $\zeta$ and $\zeta'$ have relatively prime orders, then $t = rs$
and $d=1$, giving the formula
\[
(1,1,\zeta) \otimes (1,1,\zeta') = (1,1,\zeta\zeta').
\]
Therefore, it suffices to now determine the multiplication rule when
the orders of $\zeta$ and $\zeta'$ are prime powers.

Suppose $r = p^k, s = p^l,$ and $t = p^m$, with $k \geq l$;
this forces $d = p^{k-m}$.  If $k = m$, then $t = p^k$ and $d=1$,
giving the formula
\[
(1,1,\zeta) \otimes (1,1,\zeta') = p^l (1,1,\zeta\zeta').
\]
Finally, if $k > m$, then we must have $k = l$, giving the fomula
\[
(1,1,\zeta) \otimes (1,1,\zeta') = p^m \sum_{i,j=1}^{p^{k-m}} (\eta^i,
\eta^j, \zeta\zeta'),
\]
as desired.
\end{proof}

\subsection{A Bousfield equivalent algebra to $R[H]$}

The results of the previous section allow us to compute the homotopy
groups of the $\eilm{\mb Z}$--algebra $R[H]$, which are isomorphic by
\fullref{prop:repdga} to the homology groups of $\Irr(H)$.

For convenience, we will write the order of a root of unity $\zeta$ as
$|\zeta|$.

\begin{cor}
\label{cor:irrhomology}
$\Irr(H)$ has the following homology groups:
\[
H_*(\Irr(H)) \cong \left\{
\begin{matrix}
\oplus_{\zeta \in \mu_{\infty}} \mb Z a_\zeta \hfill & \hbox{in
  dimension 0}\\
\oplus_{\zeta \in \mu_{\infty}} \left[\mb Z b_\zeta \oplus \mb Z
  c_\zeta\right] \hfill & \hbox{in dimension 1}\\
\oplus_{\zeta \in \mu_{\infty}} \mb Z d_\zeta \hfill & \hbox{in
  dimension 2}\\
0 \hfill & \hbox{otherwise.} \hfill 
\end{matrix}
\right.
\]
The products are given as follows.
\begin{eqnarray*}
a_\zeta a_{\zeta'} &=& \frac{|\zeta||\zeta'|}{|\zeta\zeta'|} a_{\zeta\zeta'}\\
a_\zeta b_{\zeta'} &=& |\zeta| b_{\zeta\zeta'}\\
a_\zeta c_{\zeta'} &=& |\zeta| c_{\zeta\zeta'}\\
a_\zeta d_{\zeta'} &=& \frac{|\zeta||\zeta\zeta'|}{|\zeta'|} d_{\zeta\zeta'}\\
b_\zeta c_{\zeta'} &=& |\zeta\zeta'| d_{\zeta\zeta'}
\end{eqnarray*}
\end{cor}

\begin{proof}
Write $a_\zeta$ for the element in $H_0(\Irr(H))$ corresponding to the
torus containing $[V_\zeta]$.  The multiplication formula for
$V_\zeta$ and $V_\zeta'$ implies the formula for $a_\zeta a_{\zeta'}$.

There are unique choices of generators for the homology of the torus
containing $[V_\zeta]$ such that $a_\zeta b_1 = |\zeta| b_\zeta$,
$a_\zeta c_1 = |\zeta| c_\zeta$, and $b_1 c_1 = d_1$, due to the
structure of $\Irr(H)$ as a set acted on by $H^*$.

The remaining product identities can all be derived from these because
the homology is torsion-free.
\end{proof}

The remainder of this section is devoted to showing that at the prime
$\ell$, much of the homotopy of $R = R[H]$ vanishes after Bousfield
localization.  Specifically, we will exhibit a subobject $\tilde R$ of $R$
such that the Bousfield localization of $R$ is the same as that of
$\tilde R$.

Consider the subspace
\[
X = \coprod_{\zeta \in \mu_{\ell^\infty}} \left(\frac{ S^1 \times S^1 }{
\langle \zeta \rangle \times \langle \zeta \rangle } \right) [V_\zeta]
\]
of $\Irr(H)$.  This subspace consists of representations in which the
central element acts by a $\ell$--power root of unity, and is closed under
$\otimes$ in the sense that the tensor product of any two
representations in $X$ is a direct sum of representations in $X$.

\begin{defn}
Define $\tilde R$ to be the subobject $\Sp(\widetilde{(\Sym^\infty
  X)})$ of $R[H]$.
\end{defn}

The map of abelian topological semirings $\Sym^\infty(X) \to
\Sym^\infty(\Irr(G)_+) \simeq \Rep(G)$ induces a map of commutative
$\mb S$--algebras $\tilde R \to R$.  On homotopy groups, this map is
the map $H_*(X) \to H_*(\Irr(H))$.

Both $R_*$ and $\tilde R_*$ have no torsion.  If we define $R_\ell =
\holim R/\ell^n$, $\pi_*(R_\ell) \cong \lim \pi_*(R)/\ell^n$, and
the map $R \to R_\ell$ gives an isomorphism after taking the smash
product $- \smsh{\eilm{\mb Z}} \eilm{\mb F_\ell}$.  By
\fullref{prop:compiso}, the map $R \to R_\ell$ is a
multiplicative $\eilm{\mb F_\ell}$--equivalence.  Similarly, we
can form $\tilde R_\ell$, and the map $\tilde R \to \tilde R_\ell$ is
a multiplicative $\eilm{\mb F_\ell}$--equivalence.

\begin{prop}
There is a map $R \to \tilde R_\ell$ of commutative $\mb S$--algebras
which is a multiplicative $\eilm{\mb F_\ell}$--equivalence.
\end{prop}

\begin{proof}
Suppose $\zeta$ is a root of unity of order relatively prime to $\ell$.
Then in $\pi_0(R_\ell)$, there is an element $a_\zeta/|\zeta|$ such that
$(a_\zeta/|\zeta|)^{|\zeta|} = 1$.  Because $|\zeta|$ is invertible in
$\pi_0(R_\ell)$, we can form the idempotent
\[
e_\zeta = \frac{1}{|\zeta|} \sum_{k=1}^{|\zeta|}
\left(\frac{a_\zeta}{|\zeta|}\right)^k = \frac{1}{|\zeta|}
\sum_{k=1}^{|\zeta|} \frac{a_{\zeta^k}}{|\zeta^k|}.
\]
This idempotent is an element of the group ring generated by
$a_\zeta/|\zeta|$ corresponding to the trivial representation.  The
kernel of multiplication by this idempotent is precisely the ideal 
generated by $1 - a_\zeta/|\zeta|$, or equivalently $|\zeta| - a_\zeta$.

The element $e_\zeta \in \pi_0(R_\ell) =
\mapset{R_\ell}{R_\ell}{R_\ell}$ can be viewed as a homotopy class of
self-maps of $R_\ell$.  We can then form the homotopy colimit
\[
e_\zeta R_\ell = \hocolim \left\{R_\ell \overto^{e_\zeta} R_\ell
\overto^{e_\zeta} R_\ell \overto^{e_\zeta} \cdots\right\},
\]
whose homotopy groups are $e_\zeta \pi_*(R_\ell)$.

By the K\"unneth spectral sequence of \cite[Theorem~IV.4.1]{ekmm97},
we have 
\[
\pi_*\left(M \smsh{R_\ell} e_\zeta R_\ell\right) \cong e_\zeta \pi_* M
\]
for any $R_\ell$--module $M$.  In particular, the map
\[
e_\zeta R_\ell \smsh{R_\ell} R_\ell \to 
e_\zeta R_\ell \smsh{R_\ell} e_\zeta R_\ell
\]
is an equivalence.  Therefore, the map $R_\ell \to e_\zeta
R_\ell$ is an $(e_\zeta R_\ell)$--equivalence of $R_\ell$--modules.
Additionally, $e_\zeta R_\ell$ is $(e_\zeta R_\ell)$--local, since any
map $f\co M \to e_\zeta R_\ell$ factors as
\[
M \overto^{\eta \smsh{} 1} e_\zeta R_\ell \smsh{R_\ell} M \overto^{1
  \smsh{} f} e_\zeta R_\ell \smsh{R_\ell} e_\zeta R_\ell \overto^\mu
e_\zeta R_\ell.
\]
Here $\eta$ is the unit and $\mu$ is the multiplication of the
$R_\ell$--ring spectrum $e_\zeta R_\ell$.  Therefore, $e_\zeta R_\ell$
is a Bousfield localization of $R_\ell$, and admits the structure of a
commutative $R_\ell$--algebra by \cite[Theorem~VIII.2.2]{ekmm97}.

We also have 
\[
\pi_*\left(\eilm{\mb F_\ell} \smsh{R_\ell} e_\zeta R_\ell\right) \cong
\pi_* \eilm{\mb F_\ell},
\]
because the idempotent $e_\zeta$ acts by 1 on $\pi_* \eilm(\mb
F_\ell)$.  Therefore, there is an isomorphism in the derived category
of $R_\ell$--modules
\[
\eilm{\mb F_\ell} \smsh{R_\ell} e_\zeta R_\ell \simeq \eilm{\mb F_\ell}.
\]
Taking homotopy colimits in $\zeta$ gives
\[
\eilm{\mb F_\ell} \smsh{R_\ell} \left(\hocolim_\zeta e_\zeta R_\ell\right)
\simeq \eilm{\mb F_\ell}.
\]
In other words, applying \fullref{prop:compiso} with base ring
$R_\ell$, the map $$R_\ell \to \hocolim e_\zeta R_\ell$$ induces an
equivalence of $\eilm{\mb F_\ell}$--localizations.

The map $\pi_*(R_\ell) \to \colim e_\zeta \pi_*(R_\ell)$ is the map that
takes the quotient by the ideal generated by all elements $a_\zeta -
|\zeta|$, where $\zeta$ varies over roots of unity relatively prime to
$\ell$.  The composite map $\pi_*(\tilde R_\ell) \to \pi_*(R_\ell) \to
\hocolim e_\zeta \pi_*(R_\ell)$ is an isomorphism.  The map $\tilde
R_\ell \to \hocolim e_\zeta R_\ell$ therefore is an equivalence.

Choosing a homotopy inverse, we can then form the composite map
\[
R \to R_\ell \to \hocolim e_\zeta R_\ell \we \tilde
R_\ell.
\]
This is a composite of multiplicative $\eilm{\mb
F_\ell}$--equivalences, and hence is a multiplicative $\eilm{\mb
F_\ell}$--equivalence by \fullref{lem:locawequiv}.
\end{proof}

\subsection{Localized homotopy of $R[H]$}
\label{sec:heisdone}

In the previous section, we found an object $\tilde R$ whose $\eilm{\mb
F_\ell}$--localization is equivalent to the localization of $R[H]$.
We will now show that the object $\tilde R_\ell = \holim \tilde
R/\ell^n$ is already $\eilm{\mb F_\ell}$--local.

\begin{thm}
$\tilde R_\ell$ is $\eilm{\mb F_\ell}$--local as an $\tilde R$--module.
\end{thm}

\begin{cor}
\label{cor:locahomology}
The homology of the $\eilm{\mb F_\ell}$--localization of $R[H]$ is
the same as the homotopy of $\tilde R_\ell$:
\[
\pi_*(\tilde R_\ell) \cong \left\{
\begin{matrix}
\widehat \oplus_{\zeta \in \mu_{\ell^\infty}} \mb Z_\ell a_\zeta \hfill &
\hbox{in dimension 0}\\
\widehat \oplus_{\zeta \in \mu_{\ell^\infty}} \left[\mb Z_\ell b_\zeta \oplus
\mb Z_\ell c_\zeta\right] \hfill & \hbox{in dimension 1}\\
\widehat \oplus_{\zeta \in \mu_{\ell^\infty}} \mb Z_\ell d_\zeta \hfill &
\hbox{in dimension 2}\\
0 \hfill & \hbox{otherwise.} \hfill 
\end{matrix}
\right.
\]
The products are given as in \fullref{cor:irrhomology}.
\end{cor}

The notation $\widehat \oplus$ denotes direct sum in the category of
$\ell$--complete modules:
\[
\widehat{\oplus} M_\alpha = \lim \oplus M_\alpha / \ell^n.
\]

\begin{proof}
Recall from \fullref{sec:completions} that $\eilm{\mb F_\ell}$--local
objects are closed under homotopy limits, $\eilm{\mb F_\ell}$--modules,
and exact triangles.  We have exact triangles
\[
\tilde R_\ell / \ell \longoverto^{\ell^n} \tilde R_\ell / \ell^{n+1} \longrightarrow
\tilde R_\ell / \ell^n
\]
for all $n$.  We will demonstrate that $\tilde R_\ell/\ell$ is $\eilm{\mb
F_\ell}$--nilpotent.  It then follows that $\tilde R_\ell/\ell^n$ is $\eilm{\mb
F_\ell}$--nilpotent for all $n$.  Because $\tilde R_\ell \simeq \holim
\tilde R_\ell/\ell^n$, $\tilde R_\ell$ is $\eilm{\mb F_\ell}$--local.

Consider the kernel $I$ of the map $\pi_0(\tilde R_\ell) \to \mb F_\ell$.  $I$ is 
generated by $\ell$ and by the elements $a_\zeta$ for $\zeta \in
\mu_{\ell^\infty}$, $\zeta \neq 1$.  In particular, the
multiplication rules show that $\ell$ divides $a_\zeta a_{\zeta'}$ for
any $\zeta, \zeta' \neq 1$, so $I^2 \subset (\ell)$.

If $M$ is any module over $\pi_0(\tilde R_\ell)$ that is acted on trivially by
$\ell$, there is a short exact sequence of $\pi_0(\tilde R_\ell / \ell)$--modules
\[
0 \to IM \to M \to M/IM \to 0.
\]
Both $IM$ and $M/ IM$ are acted on trivially by $I$.  They are
modules over $\pi_0(\tilde R_\ell)/I \cong \mb F_\ell$, and therefore the
Eilenberg--MacLane module spectra $\eilm{(IM)}$ and $\eilm{(M/IM)}$ are
local as $\tilde R_\ell$--modules.  Because local modules are closed
under exact triangles, $\eilm{M}$ is $\eilm{\mb F_\ell}$--local as an
$\tilde R_\ell$--module.

Consider the Postnikov tower for $\tilde R_\ell/\ell$, which
terminates after degree 2.
$$\xymatrix{
\ar[d] {*} & \\
(\tilde R_\ell / \ell)\langle 2\rangle \ar[d] \ar[r] & \Sigma^2
\eilm{(\pi_2(\tilde R_\ell / \ell))} \\
(\tilde R_\ell / \ell)\langle 1\rangle \ar[d] \ar[r] & \Sigma
\eilm{(\pi_1(\tilde R_\ell / \ell))} \\
(\tilde R_\ell / \ell) \ar[r] & \eilm{(\pi_0(\tilde R_\ell / \ell))} \\
}
$$
The Eilenberg--MacLane objects are modules over $\eilm{\pi_0(\tilde R_\ell)}$
and are acted on trivially by $\ell$.  They are therefore $\eilm{\mb
F_\ell}$--local as $\tilde R_\ell$--modules.  Consequently, $\tilde R_\ell / \ell$
is $\eilm{\mb F_\ell}$--local, as desired.
\end{proof}

\subsection{Derived completion of $R[H]$}
\label{sec:heisnilp}

In this section, we show that the computation of the homotopy groups
of the $\eilm{\mb F_\ell}$--localization of $R[\compl{H}_\ell]$ from the
previous section is actually a computation of the homotopy of the derived
completion of $R[\compl{H}_\ell]$, as explained in the introduction.
Consider the $\eilm{\mb F_\ell}$--nilpotent completion of the
$\mb S$--algebra $R[\compl{H}_\ell]$.  By \fullref{lem:boustocomp}
it is equivalent to the $\eilm{\mb F_\ell}$--nilpotent completion of
the $\mb S$--algebra $\tilde R$ constructed in previous sections.  Our
claim is as follows:

\begin{thm}
The pro-object $\{\tilde R/\ell^n\}$ is an $\eilm{\mb
F_\ell}$--nilpotent completion of $\tilde R$.
\end{thm}

\begin{proof}
The objects $\tilde R/\ell^n$ are $\eilm{\mb F_\ell}$--nilpotent, as
shown in \fullref{cor:locahomology}.

Let $X$ be an $\eilm{\mb F_\ell}$--module.
For any $n$, consider the following map of cofiber sequences.
$$\xymatrix{
\tilde R \ar[r]^{\ell^n} \ar[d]_\ell &
\tilde R \ar[r]^\pi \ar[d]_1 &
\tilde R/\ell^n \ar[d]_\pi \\
\tilde R \ar[r]_{\ell^{(n-1)}} &
\tilde R \ar[r]_\pi &
\tilde R/\ell^{n-1} \\
}$$
Apply $\mapset{\tilde R}{-}{X}$ to the above diagram.  Multiplication
by $\ell$ is zero on $\pi_* X$, so we get the following
map of short exact sequences.
$$\xymatrix{
0 \ar[r]&
\pi_1 X \ar[r] \ar[d]^0 &
\mapset{\tilde R}{\tilde R/\ell^{(n-1)}}{X} \ar[r] \ar[d] &
\mapset{\tilde R}{\tilde R}{X} \ar[d]^1 \ar[r] &
0 \\
0 \ar[r]&
\pi_1 X \ar[r] &
\mapset{\tilde R}{\tilde R/\ell^n}{X} \ar[r] &
\mapset{\tilde R}{\tilde R}{X} \ar[r]&
0
}$$
Taking colimits in $n$ gives the desired result.
\end{proof}

As a result, the homotopy groups of $\compl{\tilde R}_{\eilm{\mb
F_\ell}}$ are the pro-groups $\{\pi_*(\tilde R/\ell^n)\}$.  The ring
$\tilde R$ has no $\ell$--torsion, so $\pi_*(\tilde R/\ell^n) \cong
\pi_*(\tilde R)/\ell^n$, which are the same as the homotopy groups of
$\pi_*(\tilde R_\ell)$ viewed as a pro-system.

\bibliographystyle{gtart} \bibliography{link}

\begin{thebibliography}{}
\providecommand\bibmarginpar{\leavevmode\marginpar}
\def\urlstyle#1{{\tt #1}}

\bibitem{adams74}
\textbf{J\,F Adams}, \emph{Stable homotopy and generalised homology},
  University of Chicago Press, Chicago, Ill. (1974) \xox{MR}{0402720}

\bibitem{bl01}
\textbf{A Baker}, \textbf{A Lazarev},
  \href{http://dx.doi.org/10.2140/agt.2001.1.173} {\emph{On the {A}dams
  spectral sequence for $R$--modules}}, Algebr. Geom. Topol. 1 (2001) 173--199
  \xox{MR}{1823498}

\bibitem{bousfield79}
\textbf{A\,K Bousfield}, \href{http://dx.doi.org/10.1016/0040-9383(79)90018-1}
  {\emph{The localization of spectra with respect to homology}}, Topology 18
  (1979) 257--281 \xox{MR}{551009}

\bibitem{bray99}
\textbf{C Bray}, \emph{On the cohomology of some representation rings}, PhD
  thesis, Stanford University (1999)

\bibitem{carlsson78}
\textbf{G Carlsson}, \emph{An {A}dams-type spectral sequence for change of
  rings}, Houston J. Math. 4 (1978) 541--550 \xox{MR}{523612}

\bibitem{carlsson:kth}
\textbf{G Carlsson}, \emph{Structured stable homotopy theory and the descent
  problem for the algebraic $K$--theory of fields}, preprint  (2003)

\bibitem{ekmm97}
\textbf{A\,D Elmendorf}, \textbf{I Kriz}, \textbf{M\,A Mandell}, \textbf{J\,P
  May}, \emph{Rings, modules, and algebras in stable homotopy theory},
  Mathematical Surveys and Monographs 47, American Mathematical Society,
  Providence, RI (1997) \xox{MR}{1417719}

\bibitem{hps97}
\textbf{M Hovey}, \textbf{J\,H Palmieri}, \textbf{N\,P Strickland},
  \emph{Axiomatic stable homotopy theory}, Mem. Amer. Math. Soc. 128 (1997)
  x+114 \xox{MR}{1388895}

\bibitem{lubotzkymagid85}
\textbf{A Lubotzky}, \textbf{A\,R Magid}, \emph{Varieties of representations of
  finitely generated groups}, Mem. Amer. Math. Soc. 58 (1985) xi+117
  \xox{MR}{818915}

\bibitem{lydakis99}
\textbf{M Lydakis}, \href{http://dx.doi.org/10.1017/S0305004198003260}
  {\emph{Smash products and $\Gamma$--spaces}}, Math. Proc. Cambridge Philos.
  Soc. 126 (1999) 311--328 \xox{MR}{1670245}

\bibitem{mmss01}
\textbf{M\,A Mandell}, \textbf{J\,P May}, \textbf{S Schwede}, \textbf{B
  Shipley}, \href{http://dx.doi.org/10.1112/S0024611501012692} {\emph{Model
  categories of diagram spectra}}, Proc. London Math. Soc. $(3)$ 82 (2001)
  441--512 \xox{MR}{1806878}

\bibitem{nm89}
\textbf{C Nunley}, \textbf{A Magid}, \emph{Simple representations of the
  integral {H}eisenberg group}, from: ``Classical groups and related topics
  (Beijing, 1987)'', Contemp. Math. 82, Amer. Math. Soc., Providence, RI (1989)
   89--96 \xox{MR}{982280}

\bibitem{schwede99}
\textbf{S Schwede}, \href{http://dx.doi.org/10.1017/S0305004198003272}
  {\emph{Stable homotopical algebra and $\Gamma$--spaces}}, Math. Proc.
  Cambridge Philos. Soc. 126 (1999) 329--356 \xox{MR}{1670249}

\bibitem{schwede01}
\textbf{S Schwede}, \emph{$S$--modules and symmetric spectra}, Math. Ann. 319
  (2001) 517--532 \xox{MR}{1819881}

\bibitem{segal74}
\textbf{G Segal}, \href{http://dx.doi.org/10.1016/0040-9383(74)90022-6}
  {\emph{Categories and cohomology theories}}, Topology 13 (1974) 293--312
  \xox{MR}{0353298}

\bibitem{serre77}
\textbf{J-P Serre}, \emph{Linear representations of finite groups}, Springer,
  New York (1977) \xox{MR}{0450380}

\end{thebibliography}

\end{document}